\documentclass[11pt]{amsart}
\usepackage{url}
\usepackage{amsfonts, mathrsfs}
\usepackage{amssymb}
\usepackage{amsmath}
\usepackage{xr-hyper}
\usepackage{amsthm}
\usepackage[lofdepth,lotdepth]{subfig}
\usepackage{setspace,amssymb,verbatim, enumitem}
\usepackage[margin=.9in]{geometry}

\usepackage{lineno}

\usepackage{prettyref}
\newrefformat{cond}{condition (\ref{#1})}
\newrefformat{conj}{Conjecture \ref{#1}}
\usepackage{graphicx}
\usepackage{rotating}
\usepackage{psfrag}
\usepackage{afterpage}
\usepackage{multirow}
\usepackage[lofdepth,lotdepth]{subfig}
\usepackage{tikz}

\theoremstyle{plain}
\newtheorem{theorem}{Theorem}[section]
\newtheorem{proposition}[theorem]{Proposition}
\newtheorem*{theorem*}{Theorem}

\newtheorem*{corollary*}{Corollary}
\newtheorem{corollary}[theorem]{Corollary}

\newtheorem*{lemma*}{Lemma}
\newtheorem{lemma}[theorem]{Lemma}

\newtheorem{thml}{Theorem}

\newtheorem{condl}{Condition}

\theoremstyle{remark}

\newtheorem{remark}[theorem]{Remark}

\newtheorem*{case*}{Case}

\theoremstyle{definition}

\usepackage{prettyref}
\newrefformat{sec}{Section \ref{#1}}
\newrefformat{chap}{Chapter \ref{#1}}
\newrefformat{prop}{Proposition \ref{#1}}
\newrefformat{prob}{Problem \ref{#1}}
\newrefformat{rem}{Remark \ref{#1}}
\newrefformat{cor}{Corollary \ref{#1}}
\newrefformat{cond}{Condition \ref{#1}}
\newrefformat{def}{Definition \ref{#1}}
\newrefformat{ex}{Example \ref{#1}}
\newrefformat{lem}{Lemma \ref{#1}}
\newrefformat{eq}{Equation \eqref{#1}}
\newrefformat{con}{condition (\ref{#1})}
\newrefformat{conj}{Conjecture \ref{#1}}
\newrefformat{fig}{Figure \ref{#1}}
\newrefformat{tab}{Table \ref{#1}}
\newrefformat{app}{Appendix \ref{#1}}

\newcommand{\tw}[1]{{}^#1\!}

\newcommand{\Q}{\mathbb{Q}}

\newcommand{\aut}{\mathrm{Aut}}
\newcommand{\out}{\mathrm{Out}}

\newcommand{\R}{\mathbb{R}}
\newcommand{\F}{\mathbb{F}}

\newcommand{\syl}{\mathrm{Syl}}

\newcommand{\ind}{\mathrm{ind}}
\newcommand{\Ind}{\mathrm{Ind}}
\newcommand{\irr}{\mathrm{Irr}}

\newcommand{\la}{\lambda}

\newcommand{\wt}[1]{\widetilde{#1}}
\newcommand{\bG}{\mathbf{G}}
\newcommand{\bL}{\mathbf{L}}
\newcommand{\bT}{\mathbf{T}}
\newcommand{\bg}[1]{\mathbf{#1}}
\newcommand{\norm}{\mathrm{N}}
\newcommand{\cen}{\mathrm{C}}
\newcommand{\zen}{\mathrm{Z}}
\newcommand{\HC}{\mathrm{R}}

\newcommand{\SL}{\operatorname{SL}}
\newcommand{\SU}{\operatorname{SU}}
\newcommand{\type}[1]{\mathsf{#1}}
\newcommand{\Sp}{\operatorname{Sp}}
\newcommand{\gal}{\mathscr{G}}
\newcommand{\galh}{\mathscr{H}}

\newcommand{\assumption}[1]{\begin{center}
\fbox{
\parbox{6in}{\begin{center}
#1
\end{center}}
}
\end{center}}

\title[Galois-equivariant McKay bijections]{Galois-equivariant McKay bijections for primes dividing $q-1$}
\author{A. A. Schaeffer Fry}
\address{{Dept. of Mathematics and Statistics}, {Metropolitan State University of Denver}, {Denver, CO 80217}}
\email{aschaef6@msudenver.edu}

\date{}
\thanks{The author was supported in part by grants from the Simons Foundation (Award No. 351233) and the National Science Foundation (Award No. DMS-1801156).  Part of the work was completed while the author was in residence at the Mathematical
Sciences Research Institute in Berkeley, California during the Spring 2018 semester program
``Group Representation Theory and Applications", supported by the National Science Foundation
under Award No. DMS-1440140.}

\usepackage{hyperref}
\begin{document}
\maketitle

\begin{abstract}

We prove that for most groups of Lie type, the bijections used by Malle and Sp{\"a}th in the proof of Isaacs--Malle--Navarro's inductive McKay conditions for the prime $2$ and odd primes dividing $q-1$ are also equivariant with respect to certain Galois automorphisms.  In particular, this shows that these bijections are candidates for proving Navarro--Sp{\"a}th--Vallejo's recently-posited inductive Galois--McKay conditions.

\vspace{0.25cm}

\noindent \textit{Mathematics Classification Number:} 20C15, 20C33


\end{abstract}

\section{Introduction} 

One of the main sets of questions of interest in the representation theory of finite groups falls under the umbrella of the so-called local-global conjectures, which relate the character theory of a finite group $G$ to that of certain local subgroups.  Another rich topic in the area is the problem of determining how certain other groups (for example, the group of automorphisms $\aut(G)$ of $G$, or the Galois group $\mathrm{Gal}(\overline{\Q}/\Q)$) act on the set of irreducible characters of $G$ and the related question of determining the fields of values of characters of $G$.  This paper concerns the McKay--Navarro conjecture, which incorporates both of these main problems.

For a finite group $G$, the field $\Q(e^{2\pi i/|G|})$, obtained by adjoining the $|G|$th roots of unity in some algebraic closure $\overline{\Q}$ of $\Q$, is a splitting field for $G$.  The group $\gal:=\mathrm{Gal}(\Q(e^{2\pi i/|G|})/\Q)$ acts naturally on the set of irreducible ordinary characters, $\irr(G)$, of $G$ via $\chi^\sigma(g):=\chi(g)^\sigma$ for $\chi\in\irr(G), g\in G, \sigma\in\gal$.  

In \cite{navarro2004}, G. Navarro conjectured a refinement to the well-known McKay conjecture that incorporates this action of $\gal$.  Let $\ell$ be a prime and write $\galh_\ell\leqslant\gal$ for the subgroup consisting of $\sigma\in\gal$  satisfying that there is some nonnegative integer $r$ such that $\zeta^\sigma=\zeta^{\ell^r}$ for every $\ell'$ root of unity $\zeta$.  Specifically, the McKay--Navarro conjecture (also sometimes known as the ``Galois--McKay" conjecture) posits that if $P\in\syl_\ell(G)$ is a Sylow $\ell$-subgroup of $G$, then for every $\sigma\in\galh_\ell$, the number of characters in $\irr_{\ell'}(G)$ that are fixed by $\sigma$ is the same as the number of characters in $\irr_{\ell'}(\norm_G(P))$ fixed by $\sigma$.  Here for a finite group $X$, we write $\irr_{\ell'}(X)=\{\chi\in\irr(X)\mid\ell\nmid\chi(1)\}$.  

A stronger version of the McKay--Navarro conjecture says that, further, there should be an $\galh_\ell$-equivariant bijection between the sets $\irr_{\ell'}(G)$ and $\irr_{\ell'}(\norm_G(P))$.  In particular, this would imply that the corresponding fields of values are preserved over the field of $\ell$-adic numbers (see also \cite{turull}).  This version of the conjecture was recently reduced by Navarro--Sp{\"a}th--Vallejo in \cite{NavarroSpathVallejo} to proving certain inductive conditions for finite simple groups. These ``inductive Galois--McKay conditions" can roughly be described as an ``equivariant" condition and an ``extension" condition. In this article, we are concerned with the first part of \cite[Definition 3.1]{NavarroSpathVallejo} (the equivariant bijection part of the inductive Galois--McKay conditions):

\begin{condl}\label{cond:condition}

Let $G$ be a finite quasisimple group and let $Q\in\syl_\ell(G)$.  Then there is a proper $\aut(G)_Q$-stable subgroup $M$ of $G$ with $\norm_G(Q)\leqslant M$ and an $\aut(G)_Q\times\galh_\ell$-equivariant bijection 
$\irr_{\ell'}(G)\rightarrow\irr_{\ell'}(M)$ such that corresponding characters lie over the same character of $\zen(G)$.

\end{condl}

In \cite{SchaefferFrySN2S1, SFT18a, SFgaloisHC}, the author proves a consequence of the McKay--Navarro conjecture in the case that $\ell=2$.  On the way to proving this consequence, the author described in \cite{SFgaloisHC} the action of Galois automorphisms on characters of groups of Lie type, in terms of the Howlett--Lehrer parametrization of Harish-Chandra series. 
The main obstruction to the action of $\gal$ on these parameters  being well-behaved is the presence of three characters of the so-called relative Weyl group that occur in the description.  In \cite[Sections 3-4]{SFT20}, one of these characters is studied in some detail on the way to describing the fields of values for characters of symplectic and special orthogonal groups.  Here, we study the other two, especially for the principal series, which allows us to obtain our main results.


In \cite{MalleSpathMcKay2}, G. Malle and B. Sp{\"a}th complete the proof of the ordinary McKay conjecture for the prime $\ell=2$.  This groundbreaking result builds on previous work of Malle, Sp{\"a}th, and others, showing that simple groups satisfy the so-called ``inductive McKay" conditions provided by Isaacs--Malle--Navarro in  \cite[Section 10]{IsaacsMalleNavarroMcKayreduction}.  In particular, much of this work built upon \cite{malleheightzero, Malle08}, in which Malle provided a set of bijections in the case of groups of Lie type that served as candidates for the inductive McKay conditions.  In the current article, we adapt the methods and results of \cite{SFgaloisHC} to show that in most cases for the prime $\ell=2$ or odd primes dividing $q-1$, these bijections used in \cite{malleheightzero, Malle08, MalleSpathMcKay2} to prove the inductive McKay conditions for groups of Lie type defined over $\F_q$ are also $\galh_\ell$-equivariant, and therefore satisfy \prettyref{cond:condition}.  In particular, these bijections are therefore candidates for eventually proving the inductive Galois--McKay conditions.  The main result is the following:


\begin{thml}\label{thm:main}
Let $q$ be a power of a prime and assume $G$ is a group of Lie type of simply connected type defined over $\F_q$ and
such that $G$ is not of type $\type{A}_n$ nor of Suzuki or Ree type.   Let $\ell$ be a prime not dividing $q$ and write $d_\ell(q)$ for the multiplicative order of $q$ modulo $\ell$ if $\ell$ is odd, respectively modulo $4$ if $\ell=2$. Assume that one of the following holds:

\begin{enumerate}[label=(\arabic*)]
\item\label{mainodd}\label{maind=1} $d_\ell(q)=1$ and either $\ell\neq 3$ or $G$ is not $\type{G}_2(q)$ with $q\equiv 4,7\pmod 9$;
\item\label{main2} $\ell=2$, $d_\ell(q)=2$, and $G$ is not $\type{D}_n(q)$ or $\tw{2}\type{D}_n(q)$ with $n\geq 4$.
\end{enumerate}
Then \prettyref{cond:condition} holds for $G$ and the prime $\ell$.
\end{thml}

We remark that the omitted cases require additional techniques, which we plan to address in future work.  Further, the ``extension" condition in \cite[Definition 3.1]{NavarroSpathVallejo} (i.e., parts (iii) and (iv) of the relation on $\galh_\ell$-triples \cite[Definition 1.5]{NavarroSpathVallejo}) presents a separate set of challenges that we do not address here.  However, in the cases covered by \prettyref{thm:main} in which $\out(G)$ is trivial (and hence this condition is trivial) and the simple group $G/\zen(G)$ does not have an exceptional Schur multiplier, we obtain as an immediate corollary that $G/\zen(G)$ satisfies the inductive Galois--McKay conditions.  
In particular, \prettyref{thm:main} yields that for $q=p$ a prime, the groups $\type{G}_2(p)$ with $p>3$ and $\type{F}_4(p)$ and $\type{E}_8(p)$ with $p$ odd satisfy the inductive Galois--McKay conditions for all primes $\ell$ that divide $p-1$, with the possible exception of $\ell=3$ for $\type{G}_2(p)$.

%
%
%
%

Setting aside for the moment the inductive conditions, even the statement of the McKay--Navarro conjecture has only been proved for a very limited number of simple groups, most notably the case of groups of Lie type in defining characteristic \cite{ruhstorfer}.  On the way to proving \prettyref{thm:main}, we show that if $G$ of type $\type{B}$ or is exceptional but not of type $\type{E}_6, \tw{2}\type{E}_6$, nor Suzuki or Ree type, then all odd-degree characters of $G$ are rational-valued.  As a consequence, we obtain the statement of the McKay--Navarro conjecture for these groups.

\begin{thml}\label{thm:galmck}
Let $q$ be a power of an odd prime and let $S$ be a simple group $\type{G}_2(q)$, $\tw{3}\type{D}_4(q)$, $\type{F}_4(q),$  $\type{E}_7(q)$, $\type{E}_8(q)$, 
 $\type{B}_n(q)$ with $n\geq3$, or $\type{C}_n(q)$ with $n\geq 2$ and $q\equiv \pm1\pmod 8$.  Then the McKay--Navarro conjecture holds for $S$ for the prime $\ell=2$.
 \end{thml}

The structure of the paper is as follows.  In \prettyref{sec:prelim}, we set some notation and hypotheses to be used throughout the paper, and recall some previous results from \cite{MalleSpathMcKay2} about odd-degree characters and from \cite{SFgaloisHC} regarding the action of $\gal$ on characters of groups of Lie type.  In Sections \ref{sec:gamma}-\ref{sec:genericalg}, we discuss in some detail the three characters of the relative Weyl group that arise in the description of this action as mentioned above.  
In \prettyref{sec:principalseries},  we prove \prettyref{thm:main}\ref{mainodd} with the exception of $\ell=2$ and $G=\Sp_{2n}(q)$ for certain $q$.  In \prettyref{sec:typeC}, we complete the proof of \prettyref{thm:main} when $\ell=2$ and  $G=\Sp_{2n}(q)$.  Finally in \prettyref{sec:typeB}, we prove \prettyref{thm:main}\ref{main2} for the remaining cases, as well as \prettyref{thm:galmck}.

\section{Additional Notation and Preliminaries}\label{sec:prelim}

In this section, we set some notation to be used throughout the paper, before recalling the main result of \cite[Section 3]{SFgaloisHC} and some key ingredients from \cite{MalleSpathMcKay2}.  

\subsection{Basic Notation}

Let $X\leqslant Y$ be a subgroup of a finite group $Y$. For an irreducible character $\chi\in\irr(Y)$, we denote by $\chi\downarrow X$ the restriction of $\chi$ to a character of $X$.  For $\varphi\in\irr(X)$, we denote by $\Ind_X^Y(\varphi)$ the induced character of $Y$. The notations $\irr(X|\chi)$ for $\chi\in\irr(Y)$ and $\irr(Y|\varphi)$ for $\varphi\in\irr(Y)$ will denote the set of irreducible constituents of $\chi\downarrow X$, respectively of $\Ind_X^Y(\varphi)$. For a prime $\ell$, the notation $\irr_{\ell'}(\bullet)$ will always mean the subset of $\irr(\bullet)$ consisting of characters with degree relatively prime to $\ell$.

If $X\lhd Y$ is normal and $\varphi\in\irr(X)$, we will denote by $Y_\varphi$ the stabilizer (also known as the inertia subgroup) of $\varphi$ in $Y$.  As in \cite{Spath09}, if $\mathcal{X}$ is a subset of $\irr(X)$ such that every $\varphi\in\mathcal{X}$ extends to its inertia group $Y_\varphi$ in $Y$, we define an \emph{extension map for $\mathcal{X}$ with respect to $X\lhd Y$} to be any map $\mathcal{X}\rightarrow \bigcup_{\varphi\in\mathcal{X}} \irr(Y_\varphi)$ that sends each $\varphi\in\mathcal{X}$ to one of its extensions to $Y_\varphi$.

If $Y$ is a cyclic group, we will denote by $-1_Y$ the unique character of $Y$ of order $2$.

\subsection{Groups of Lie Type and Relevant Subgroups}

Let $q=p^a$ be a power of a prime $p$. Throughout, unless otherwise specified, we let $\bG$ be a simple, simply connected algebraic group defined over $\overline{\mathbb{F}}_q$.


 Let $\bT\leqslant\bg{B}$ be a maximal torus and Borel subgroup of $\bG$, respectively. Let $\Phi$ be the root system of $\bg{G}$ with respect to $\bg{T}\leqslant \bg{B}$ and $\Delta$ a set of fundamental roots.  We will at times do computations in $\bG$ using the Chevalley generators and relations, as in \cite[Theorem 1.12.1]{gorensteinlyonssolomonIII}.  In particular, $x_\alpha(t), n_\alpha(t),$ and $h_\alpha(t)$ for $t\in\overline{\F}_q$ and $\alpha\in\Phi$ are as defined there, and $\bT$ may be written as the direct product of groups $\prod_{\alpha\in\Delta}\{h_\alpha(t)\mid t\in\overline{\F}_q^\times\}$ since $\bG$ is simply connected. 
 
Now let $G=\bG^F$ be the group of fixed points of $\bG$ under a Frobenius endomorphism $F$ defined over $\F_q$.  That is, $F\colon \bG\rightarrow \bG$ is of the form $F_q\circ \tau$, where $\tau$ is a graph automorphism defined by $\tau\colon x_\alpha(t)\mapsto x_{\tau'(\alpha)}(t)$ for $\alpha\in\Delta, t\in \overline{\F}_q$, and $\tau'$ a length-preserving symmetry of the Dynkin diagram associated to $\Delta$; and $F_q=F_p^a$ where $F_p\colon \bG\rightarrow \bG$ is defined by $F_p: x_\alpha(t)\mapsto x_\alpha(t^p)$ for $t\in \overline{\F}_q$ and $\alpha\in\Phi$. (In particular, we omit the cases of Suzuki and Ree groups.)  Then $\bT$ and $\bg{B}$ are $F$-stable.

 As in \cite[Setting 2.1]{Spath09} or \cite[Section 3.A]{MalleSpathMcKay2}, we let $V:=\langle n_\alpha(\pm1)\mid\alpha\in\Phi\rangle\leqslant \norm_{\bg{G}}(\bg{T})$ be the \emph{extended Weyl group} and let $H$ be the group $H:=V\cap \bg{T}=\langle h_\alpha(-1)\mid\alpha\in\Phi\rangle$, which is elementary abelian of size $(2,q-1)^{|\Delta|}$.   Then we have $\norm_{\bG}(\bT)=\langle \bT, V\rangle$.  Let $\bg{W}:=\norm_{\bG}(\bT)/\bT$ be the Weyl group of $\bG$. (Note that $\bg{W}=V$ and $H=1$ if $p=2$.)  Throughout, we further write $T:=\bg{T}^F$ and $N:=\norm_{\bg{G}}(\bg{T})^F$.

Now, we let $v\in\bg{G}$ be the canonical representative in $V$ of the longest element of the Weyl group $\bg{W}$ of $\bg{G}$, as in \cite[Section 3.A]{MalleSpathMcKay2} or \cite[Definition 3.2]{spath10}.  Note that a suitable conjugation in $\bg{G}$ induces an automorphism of $\bg{G}$ mapping $G$ to $\bg{G}^{vF}$.  Write $T_1=\bg{T}^{vF}, V_1:=V^{vF}$, and  $H_1:=H^{vF}$.  Then by \cite[Lemma 3.2]{MalleSpathMcKay2} (see also \cite[Proposition 6.4 and Table 2]{Spath09} for the case $\tw{3}\type{D}_4(q)$ and \cite[Proposition 5.1 and Lemma 6.1]{Spath09} in the cases $\type{G}_2(q), \type{F}_4(q)$, and $\type{E}_8(q)$), we have $\bg{T}=\cen_{\bg{G}}(\bg{S})$ for some Sylow $2$-torus $\bg{S}$ of $(\bg{G}, vF)$, and we have $N=TV$ and $N_1=T_1V_1$, where $N_1:=\norm_{\bg{G}}(\bg{S})^{vF}$.  Further, \cite[Proposition 5.11]{malleheightzero} yields that $N$ controls $G$-fusion in $T$ and $N_1$ controls $\bg{G}^{vF}$-fusion in $T_1$, and the results of \cite{malleheightzero} imply that in most cases, the groups $N$ and $N_1$ are the natural choices to play the role of $M$ in \prettyref{cond:condition}.

We note the following, which follows from the fact that the longest element of $\bg{W}$ is central in the stated cases (see the proofs of \cite[Lemma 6.1]{Spath09} and \cite[Lemma 3.2]{MalleSpathMcKay2}).

\begin{lemma}\label{lem:V1V}
Assume $\bg{G}$ is of type $\type{B}_n, \type{C}_n, \type{D}_{2n}, \type{G}_2, \type{F}_4, \type{E}_7, \type{E}_8$ or that $G=\tw{2}{\type{D}_n}(q)$ or $\tw{3}{\type{D}_4}(q)$.   Then $V_1=V$ and $H_1=H$.
\end{lemma}

If $\ell$ is a prime different than $p$, we will write $d_\ell(q)$ for the order of $q$ modulo $\ell$ if $\ell$ is odd, and the order of $q$ modulo $4$ if $\ell=2$.  The role of $T_1$ and $N_1$ will be important in the case $d_2(q)=2$, i.e. $q\equiv 3\pmod 4$. 

\subsection{Galois Automorphisms and Harish-Chandra Series}\label{sec:HCseries}

Let $\bg{L}\leqslant \bg{P}$ be an $F$-stable Levi subgroup and parabolic subgroup of $\bg{G}$, respectively, so that $L:=\bL^F$ is a split Levi subgroup in the parabolic subgroup $P:=\bg{P}^F$ of $G=\bG^F$.   Let $\la\in\irr(L)$ be a cuspidal character and write $W(\la):=N(L)_\la/L$, where $N(L):=(\norm_G(L)\cap N)L$. Recall that such a pair $(L, \la)$ is called a cuspidal pair for $G$, and the Harish-Chandra induced character $\HC_L^G(\la)$ is defined as $\Ind_P^G(\mathrm{Inf}_L^P(\la))$, where $\mathrm{Inf}_L^P(\la)$ denotes the inflation of $\la$ to a character of $P$, viewing $L$ as a quotient of $P$ by its unipotent radical.  We will write $\mathcal{E}(G, L,\la)\subseteq\irr(G)$ for the set of constituents of $\HC_L^G(\la)$, known as a Harish-Chandra series.

Using the fact that every cuspidal $\psi\in\irr(L)$ extends to $N(L)_\psi$ due to \cite{geck93} and \cite[Theorem 8.6]{lusztig84}, we will apply the concept of extension maps to the case of cuspidal characters of $L$ and $L\lhd N(L)$.  Thanks to this and the work of Howlett and Lehrer \cite{HowlettLehrer80, howlettlehrer83}, for a cuspidal pair $(L, \la)$ of $G$, we have a bijection between  $\mathcal{E}(G, L, \la)$ and $\irr(W(\la))$ induced by a bijection $\mathfrak{f}\colon \irr(\mathrm{End}_{\overline{\Q}G}(\HC_T^G(\la))\rightarrow \irr(W(\la))$, where by $\mathrm{End}_{\overline{\Q}G}(\HC_T^G(\la))$ we mean the endomorphism algebra of a fixed module affording $\HC_T^G(\la)$.  We write $\HC_L^G(\la)_\eta$ for the character of $G$ in $\mathcal{E}(G, L, \la)$ corresponding to $\eta\in\irr(W(\la))$.  In particular, every character of $G$ can be written in such a way.

Now, there is a subsystem $\Phi_\la\subseteq \Phi$ with simple roots $\Delta_\la\subseteq \Phi_\la\cap \Phi^+$ such that $W(\la)$ can be decomposed as a semidirect product $R(\la)\rtimes C(\la)$, where $R(\la)=\langle s_\alpha\mid \alpha\in\Phi_\la\rangle$ is a Weyl group with root system $\Phi_\la$ and $C(\la)$ is the stabilizer of $\Delta_\la$ in $W(\la)$. (See \cite[Section 10.6]{carter2} or \cite[Section 2]{HowlettLehrer80} for more details.) Here for $\alpha\in\Phi$, $s_\alpha$ is the corresponding reflection in $\bg{W}^F=N/T$ induced by the element $n_\alpha(-1)$ of $N$.  The root system $\Phi_\la$ is obtained as follows.  Set

\[\Phi^\flat=\{\alpha\in\Phi\setminus\Phi_L\mid w(\Delta_L\cup\{\alpha\})\subseteq\Delta\hbox{ for some $w\in W$ and }(w_0^Lw_0^{\alpha})^2=1\}.\]  Here $w_0^L, w_0^\alpha$ are the longest elements in $W(L):=N(L)/L$ and $\langle W(L), s_\alpha\rangle$, respectively, and $\Phi_L\subseteq\Phi$ is the root system of $W(L)$ with simple system $\Delta_L\subseteq\Delta$.  Then for $\alpha\in\Phi^\flat$, letting $L_\alpha$ denote the standard Levi subgroup of $G$ with simple system $\Delta_L\cup\{\alpha\}$, $L$ is a standard Levi subgroup of $L_\alpha$ and  $p_{\alpha,\la}\geq1$ is defined to be the ratio between the degrees of the two constituents of $\HC_L^{L_\alpha}(\la)$.  Then \[\Phi_{\la}=\{\alpha\in\Phi^\flat\mid s_\alpha\in W(\la) \hbox{ and } p_{\alpha,\la}\neq1\}.\]  

 Theorem 3.8 of \cite{SFgaloisHC} describes the  action of $\gal$ on $\irr(G)$ with respect to the parameters $(L, \la, \eta)$.   We restate it here for the convenience of the reader, although we first must establish the necessary notation. 
 
\textit{Notation for \prettyref{thm:GaloisAct}} Let $(L, \la)$ be a cuspidal pair for $G$. 
 Fix an extension map $\Lambda$ for cuspidal characters with respect to $L\lhd N(L)$, so that for each cuspidal $\psi\in\irr(L)$, we have $\Lambda(\psi)$ is an extension of $\psi$ to $N(L)_\psi$.  For $\sigma\in\gal$, note that $N(L)_\la=N(L)_{\la^\sigma}$, and define $\delta_{\la,\sigma}$ to be the linear character of $W(\la)$ such that $\Lambda(\la)^\sigma=\delta_{\la,\sigma}\Lambda({\la^\sigma})$, guaranteed by Gallagher's theorem \cite[Corollary 6.17]{isaacs}.  Further, let $\delta'_{\la,\sigma}\in\irr(W(\la))$ be the character such that $\delta'_{\la,\sigma}(w)=\delta_{\la,\sigma}(w)$ for $w\in C(\la)$ and $\delta'_{\la,\sigma}(w)=1$ for $w\in R(\la)$.  Let $\gamma_{\la, \sigma}$ be the function on $W(\la)$ such that $\gamma_{\la, \sigma}(w)=\frac{\sqrt{\ind(w_c)}^\sigma}{\sqrt{\ind(w_c)}}$ where $w=w_rw_c$ for $w_c\in C(\la)$ and $w_r\in R(\la)$.  
 
 For $\eta\in\irr(W(\la))$, we denote by $\eta^{(\sigma)}$ the character $\mathfrak{f}(\wt{\eta}^\sigma)$, where $\wt{\eta}\in\irr(\mathrm{End}_{\overline{\Q}G}(\HC_T^G(\la)))$ is such that $\mathfrak{f}(\wt{\eta})=\eta$.  (See \cite[Section 3.5]{SFgaloisHC}.) Note that this is not necessarily the same as $\eta^\sigma$, although we show below in \prettyref{sec:genericalg} that this is often the case.

\begin{theorem}{\cite[Theorem 3.8]{SFgaloisHC}}\label{thm:GaloisAct}
Let $\sigma\in\gal$, let $(L, \la)$ be a cuspidal pair for $G$, and keep the notation above. 
Let $\eta\in \irr(W(\la))$.  Then

\[\left(\HC_L^G(\la)_\eta\right)^\sigma=\HC_L^G(\lambda^\sigma)_{\eta'},\] where $\eta'\in\irr(W(\la))=\irr(W(\la^\sigma))$ is defined by $\eta'(w)=\gamma_{\la, \sigma}(w)\delta'_{\la,\sigma}(w^{-1})\eta^{(\sigma)}(w)$ for each $w\in W(\lambda)$.
\end{theorem}

\prettyref{thm:GaloisAct} shows that the characters $\gamma_{\la,\sigma}$, $\delta'_{\la,\sigma}$, and $\eta^{(\sigma)}$ are the obstructions to $\gal$ acting on the parameters $(L, \la, \eta)$ in an equivariant way.  We discuss these characters in some detail below in Sections \ref{sec:gamma}, \ref{sec:extnmap}, and \ref{sec:genericalg}, respectively.

\subsection{Lusztig Series and Odd-Degree Characters}
We will write $G^\ast$ for the dual $G^\ast={\bG^\ast}^{F^\ast}$ of $G$, where $(\bG^\ast, F^\ast)$ is dual to $(\bG, F)$ as in, for example, \cite[Section 4.2]{carter2}.   Let $\bT^\ast$ be an $F^\ast$-stable maximal torus of $\bG^\ast$ dual to $\bT$.  Write $T=\bT^F$ and $T^\ast=(\bT^\ast)^{F^\ast}$.   We write $W=\bg{W}^F$, where $\bg{W}=\norm_\bg{G}(\bg{T})/\bg{T}$, and similarly for $W^\ast$.  This duality induces an isomorphism $\irr(T)\rightarrow T^\ast$ (see, for example, \cite[Proposition 4.4.1]{carter2}).  For $\la\in\irr(T)$, the following lemma allows us to view $W(\la)$ and $R(\la)$ as the groups $W(s)$ and $W^\circ(s)$ for a certain semisimple element $s$, where we write $W(s)$ and $W^\circ(s)$ for the fixed points of the Weyl groups of $\cen_{\bG^\ast}(s)$ and $\cen_{\bG^\ast}^\circ(s)$, respectively, under $F^\ast$, and implies that $C(\la)$ is isomorphic to a subgroup of $\left(\zen(\bG)/\zen(\bG)^\circ\right)^F$.  (See also \cite[Lemma 3.8]{SFT20}.)

\begin{lemma}\label{lem:W(s)}{\cite[Lemma 4.5]{SFgaloisHC} }
Let $\la\in\irr(T)$ and let $s\in T^\ast$ correspond to $\lambda$ in the sense of \cite[Proposition 4.4.1]{carter2}.   Then 
 $W(\la)$ is isomorphic to $W(s)$. 
Further, if $\bG$ is simple of simply connected type, not of type $\type{A}_n$, then there is an isomorphism $\kappa\colon W(\la)\rightarrow W(s)$ such that $\kappa(R(\la))=W^\circ(s)$.  In particular, in this case $W(\la)/R(\la)$ is isomorphic to $(\cen_{\bG^\ast}(s)/\cen_{\bG^\ast}^\circ(s))^{F^\ast}$.
\end{lemma}

Recall that $\irr(G)$ may be partitioned into so-called \emph{rational Lusztig series} $\mathcal{E}(G, s)$ ranging over the $G^\ast$-classes of semisimple elements $s$ in $G^\ast$.  Each $\mathcal{E}(G,s)$ is further a disjoint union of Harish-Chandra series $\mathcal{E}(G, L,\la)$ satisfying $\la\in\mathcal{E}(L, s)$ (see \cite[11.10]{bonnafe06}).  The following results from \cite{MalleSpathMcKay2} describe the Lusztig and Harish-Chandra series containing odd-degree characters.

\begin{theorem}{\cite[Theorem 7.7]{MalleSpathMcKay2}}\label{thm:MS7.7}
Let $\bG$ be simple, of simply connected type, not of type $\type{A}_n$.  Let $\chi \in\irr_{2'}(G)$. Then either $\chi$ lies in the principal series of
$G$, or $q\equiv 3\pmod{4}$, $G=\Sp_{2n}(q)$ with $n\geq 1$ odd, $\chi\in\mathcal{E}(G,s)$ with $\cen_{G^\ast}(s)=\type{B}_{2k}(q)\cdot\tw{2}{\type{D}}_{n-2k}(q).2$, where $0\leq k\leq (n-3)/2$ and $\chi$ lies in the Harish-Chandra series of a cuspidal character of degree $(q-1)/2$ of a Levi subgroup $\Sp_2(q)\times (q-1)^{n-1}$.\end{theorem}


\begin{lemma}{\cite[extension of Lemma 7.9]{MalleSpathMcKay2}}\label{lem:MS7.9}
 Let $\bG$ be simple, of simply connected type, not of type $\type{A}_n$. Let $\chi\in\irr_{2'}(G)$.  Then $\chi=\HC_T^G(\lambda)_\eta$, where $T$ is a maximally split torus of $G$, $\lambda\in\irr(T)$ is such that $2\nmid[W:W(\lambda)]$, and $\eta\in\irr_{2'}(W(\lambda))$, except possibly in the case  $\bG$ is type $\type{C}_n$ with $n$ odd and $q\equiv3\pmod{4}$.  In the latter case, $\chi$ may also be of the form $\chi=\HC_L^G(\lambda)_{\eta}$ with $(L,\lambda)$ as in \prettyref{thm:MS7.7}, $2\nmid[W(L):W(\lambda)]$ and $\eta\in\irr_{2'}(W(\lambda))$.
\end{lemma}


The following is \cite[Lemma 7.5]{MalleSpathMcKay2}, extended to some additional exceptional groups.
\begin{lemma}\label{lem:sinvol}
Let $G=\bG^F$ where $\bG$ is simple of simply connected type, not of type $\type{A}_n$, $\type{C}_n$, $\type{E}_6$, or $\type{D}_{2n+1}$, and such that $F$ is a Frobenius endomorphism defining $G$ over $\F_q$ with $q$ odd.  Then if $\chi\in\irr_{2'}(G)$ lies in the rational series $\mathcal{E}(G,s)$, then $s^2=1$. 
\end{lemma}
\begin{proof}
Recall that $\chi\in\irr_{2'}(G)$ implies that $s$ centralizes a Sylow $2$-subgroup of $G^\ast$.  Then for types $\type{B}_n, \type{D}_{2n}, \type{E}_7$, this is \cite[Lemma 7.5]{MalleSpathMcKay2}.  For the remaining cases, the statement can be obtained by analyzing the possible centralizer structures of $G^\ast$, which are listed for example at \cite{luebeckwebsite}, or in \cite{deriziotismichler} for type $\tw{3}\type{D}_4$.  In particular, if $s$ is non-central, then we have $2\mid [W:W(s)]$ unless $\cen_{G^\ast}(s)$ is of type
$\type{A}_1(q)+\type{A}_1(q)$ if $G=\type{G}_2(q)$; $\type{B}_4(q)$ if  $G=\type{F}_4(q)$;
 $\type{D}_8(q)$ if $G=\type{E}_8(q)$;
and $\type{A}_1(q^3)+\type{A}_1(q)$ if $G=\tw{3}\type{D}_4(q)$, in which cases $s^2=1$.
\end{proof}

\section{Square Roots and the Character $\gamma_{\la,\sigma}$}\label{sec:gamma}

Let $p$ be a prime and $q$ a power of $p$.  Keep the notation of \prettyref{thm:GaloisAct}. The following, found as \cite[Lemma 3.11]{SFT20}, shows that $\gamma_{\la,\sigma}$ is indeed a character of $W(\la)$ (and hence of $C(\la)$) and is closely related to the action of $\sigma\in\gal$ on $\sqrt{p}$.  We remark that here we do not require the assumption that $\bG$ is simply connected.

\begin{lemma}\label{lem:SFT20Lemma3.11}
Assume $(L, \la)$ is a cuspidal pair for a finite reductive group $G=\bG^F$ defined over $\F_q$ with $q$ a power of a prime $p$, and let $\sigma\in\gal$. Then in the notation of \prettyref{thm:GaloisAct},
$\gamma_{\la,\sigma}$ is a character of $W(\la)$, and is moreover trivial if and only if at least one of the following holds:
\begin{itemize}
\item $q$ is a square;
\item the length $l(w_c)$, with respect to the Weyl group of $\bG$, is even for each $w_c\in C(\la)$; or
\item $\sigma$ fixes $\sqrt{p}$.
\end{itemize}
Otherwise, $\gamma_{\la,\sigma}(w)=(-1)^{l(w_c)}$, where $w=w_rw_c$ with $w_r\in R(\la)$ and $w_c\in C(\la)$.
\end{lemma}

Let $\ell\neq p$ be another prime. 
In \cite[Section 4]{SFT20}, we discuss the action of $\galh_\ell$ on $\sqrt{p}$ (and on $\gamma_{\la,\sigma}$) when $p$ is odd.  We record the following observations, which extend the discussion to the case $p=2$.

\begin{lemma}\label{lem:sqrt}
Let $\ell$ be an odd prime and let $\sigma\in\galh_\ell$ with corresponding integer $r$ such that  $\sigma(\zeta)=\zeta^{\ell^r}$ for every $\ell'$ root of unity $\zeta$.   Let $p\neq \ell$ be a prime.  Then $\sqrt{p}^\sigma=\sqrt{p}$ if and only if $r$ is even or $p$ is a square modulo $\ell$.
\end{lemma}

\begin{proof}
For $p$ odd, this is \cite[Lemma 4.4]{SFT20}.
Note that $\sqrt{2}$ may be written $\sqrt{2}=\zeta_8+\zeta_8^{-1}$ for some fixed primitive $8$th root of unity $\zeta_8$. Then $\sqrt{2}^\sigma=\sqrt{2}$ if and only if $\ell^r$ is $\pm 1\pmod 8$.  This proves the claim, since the condition $\ell\equiv \pm1\pmod 8$ is equivalent to $2$ being square modulo $\ell$.
\end{proof}

\begin{corollary}\label{cor:maxsplitsqrt}
Keep the hypothesis of \prettyref{lem:SFT20Lemma3.11}, and  assume that $\ell\neq p$ is an odd prime dividing $q-1$.  Then $\Q(\sqrt{q})$ is contained in the fixed field of $\galh_\ell$ and the character $\gamma_{\la,\sigma}$ is trivial. 
\end{corollary}
\begin{proof}

This follows from Lemmas \ref{lem:SFT20Lemma3.11} and \ref{lem:sqrt}, since the condition $q\equiv 1\pmod\ell$ forces $p$ to be a square modulo $\ell$ if $q$ is not a square.
\end{proof}


 
 In the case $\ell=2$, we have the following, returning to our default assumption that $\bG$ is simply connected.
 \begin{proposition}\label{prop:Clambdaeven}
 Let $G=\bG^F$ where $\bG$ is simple of simply connected type, not of type $\type{A}_n$ or $\type{C}_n$, and such that $F$ is a Frobenius endomorphism defining $G$ over $\F_q$ with $q$ odd. Let $\chi=\HC_T^G(\la)_\eta\in\irr_{2'}(G)$, as in \prettyref{lem:MS7.9}.  Then $\gamma_{\la,\sigma}=1$ for every $\sigma\in\gal$.
 \end{proposition}
 \begin{proof}
 If $|\zen(\bG)^F|$ is odd, this follows from \prettyref{lem:W(s)} and the fact that $\gamma_{\la,\sigma}$ has order dividing $2$.  If $\bG$ is type $\type{B}_n$ ($n\geq 3$) or $\type{D}_n$ ($n\geq 4$) and $G\neq \tw{3}{\type{D}_4}(q)$, then every member of $C(\la)$ has even length in the Weyl group of $\bG$, by \cite[Proposition 4.11]{SFgaloisHC}, and hence the statement holds using \prettyref{lem:SFT20Lemma3.11}.  This leaves the case $\bG$ is of type $\type{E}_7$. 
 As in \prettyref{lem:sinvol}, we may analyze the possible centralizer structures and see that either $s=1$ or $\cen_{\bG^\ast}(s)$ is connected, of type $\type{D}_6+\type{A}_1$.  In either case, this again yields $C(\la)=1$, and the result follows.
 \end{proof}

\section{The Extension Map $\Lambda$}\label{sec:extnmap}

Keep the notation of \prettyref{sec:prelim}.  In particular, recall that $G=\bG^F$ is of  simply connected type, not of Suzuki or Ree type, defined over $\F_q$ with $q$ a power of the prime $p$, and we have $V:=\langle n_\alpha(\pm1)\mid\alpha\in\Phi\rangle\leqslant \norm_{\bg{G}}(\bg{T})$, $H:=V\cap\bg{T}$ is an elementary abelian $2$-group (and trivial when $p=2$), $N=TV$, and $N_1=T_1V_1$.   For $\la\in\irr(T)$, note that $N_\la=TV_\la$.

Let $\iota\colon \bG\hookrightarrow \wt{\bG}$ be a regular embedding such that  $F$ extends to a Frobenius endomorphism of $\wt{\bG}$, which we will also denote by $F$. (See e.g. \cite[Section 1.7]{GM20} for more details on regular embeddings.) We will write $\wt{G}:=\wt{\bG}^F$. Let $D$ denote the subgroup of outer automorphisms generated by the field automorphism $F_0$ and graph automorphisms commuting with $F$, so that $\aut(G)$ is induced by $\wt{G}\rtimes D$.

Our goal in this section is to obtain an extension map $\Lambda$ with respect to $T\lhd N$ that is ``close to" $\gal$-equivariant.  In particular, we make use of the existing maps from \cite{Spath09, MalleSpathMcKay2}.  We begin with the case of some exceptional groups.


\begin{lemma}\label{lem:exceptLambda}
Let $G$ be as above, and assume either $p=2$ or $\bG$ is of type $\type{G}_2, \type{F}_4, \type{E}_6$, or $\type{E}_8$.  Then there is an ${N}D\times\gal$-equivariant extension map $\Lambda$ with respect to $T\lhd N$. 
\end{lemma}
\begin{proof}
The construction of the map $\Lambda$ is exactly as in \cite[Corollary 3.13]{MalleSpathMcKay2} in the case $d=v=1$, using the construction from \cite[Section 5]{Spath09} for an extension map with respect to $H\lhd V$.  We provide the construction for the convenience of the reader, and to show that it is $\gal$-equivariant.

\emph{(1) A $\gal$-equivariant extension with respect to $H\lhd V$.}  First, note that when $p=2$, we have $H=1$.  Hence we may assume $p$ is odd. According to \cite[Lemma 5.3]{Spath09}, for each $\delta\in\irr(H)$, there is a subset $R(\delta)$ of $\Phi$ such that $V_\delta=\langle H, n_\alpha(-1)\mid \alpha\in R(\delta)\rangle$.  By \cite[Proof of Lemma 5.4]{Spath09} there is an extension map $\Lambda_0'$ with respect to $H\lhd V$ such that $\Lambda_0'(\delta)$ is trivial on $n_\alpha(-1)$ for each $\alpha\in R(\delta).$  Note that $H$ is an elementary abelian $2$-group, so $\delta^\sigma=\delta$ for all $\delta\in\irr(H)$ and all $\sigma\in\gal$.  Hence we see that \[\Lambda_0'(\delta)^\sigma=\Lambda_0'(\delta)=\Lambda_0'(\delta^\sigma)\] for each $\delta\in\irr(H)$ and $\sigma\in\gal$.

\emph{(2) A $VD\times \gal$-equivariant extension with respect to $H\lhd V$.}  Let $\delta_1,\ldots,\delta_r$ be representatives of the $V$-orbits on $\irr(H)$.  For each $x\in V$ and $i\in\{1,...,r\}$, there is some linear $\mu_{x,i}\in\irr(V_{\delta_i}/H)$ such that $\Lambda_0'(\delta_i)^x=\mu_{x,i}\Lambda_0'(\delta_i^x)$.  Note that for each $x$ and $i$, $\mu_{x,i}$ is fixed by each $\sigma\in\gal$, since $V_{\delta_i}/H$ is generated by involutions and $\mu_{x,i}$ is a linear character.  Then we define another extension map with respect to $H\lhd V$ via $\Lambda_0(\delta_i):=\Lambda_0'(\delta_i)$ for each $i$ and \[\Lambda_0(\delta_i^x):=\Lambda_0'(\delta_i)^x=\mu_{x,i}\Lambda_0'(\delta_i^x)\] for each $x\in V$.   This map is well-defined, since if $\delta_i^x=\delta_i^y$, then $xy^{-1}\in V_{\delta_i}$ fixes $\delta_i$ and $\Lambda_0'(\delta_i)$.  By construction, $\Lambda_0$ is $V$-equivariant.  Further, we have for each $x\in V$, $i\in\{1,\ldots,r\}$, and $\sigma\in\gal$,
\[\Lambda_0(\delta_{i}^x)^\sigma=\mu_{x,i}^\sigma\Lambda_0'(\delta_{i}^x)^\sigma=\mu_{x,i}\Lambda_0'(\delta_i^{x})=\Lambda_0(\delta_i^x).\]  

Hence $\Lambda_0$ is a $V$-equivariant extension with respect to $H\lhd V$ which still satisfies $\Lambda_0(\delta)^\sigma=\Lambda_0(\delta)=\Lambda_0(\delta^\sigma)$ for each $\delta\in\irr(H)$ and $\sigma\in\gal$.  Further, since $F_0$ acts trivially on $V$, the extension map $\Lambda_0$ is $VD\times\gal$-equivariant by the proofs of \cite[Proposition 3.9]{MalleSpathMcKay2} and \cite[Lemma 8.2]{Spath09}.

\emph{(3) The $ND\times \gal$-equivariant extension with respect to $T\lhd N$.}   Finally, we use the strategy of \cite[4.2]{Spath09} to obtain the desired extension map with respect to $T\lhd N$.  For $\lambda\in\irr(T)$, we define an extension $\Lambda(\la)=\Lambda_\la$ of $\la$ to $N_\la$ via $\Lambda_\la(vt)=\Lambda_0(\la\downarrow_H)(v)\la(t)$ for each $v\in V_\la$ and $t\in T$, since $N=TV$.  Then $\Lambda$ is $ND$-equivariant by construction, and
\[\Lambda_\la^\sigma(vt)=\Lambda_0(\la\downarrow_H)^\sigma(v)\la^\sigma(t)=\Lambda_0(\la\downarrow_H^\sigma)(v)\la^\sigma(t)=\Lambda_{\la^\sigma}(vt)\] for each $\sigma\in\gal$, $v\in V_\la$, and $t\in T$.
\end{proof}

\begin{remark}\label{rem:trialityLambda}
When $G=\tw{3}\type{D}_4(q)$, we similarly have an $ND$-equivariant extension map $\Lambda$ with respect to $T\lhd N$, by \cite[Theorem 3.6]{CS13}. 
\end{remark}


%

\assumption{
For the remainder of the article, we let $\Lambda$ denote the $ND$-equivariant extension map with respect to $T\lhd N$ as in \prettyref{lem:exceptLambda}, \prettyref{rem:trialityLambda}, or \cite[Corollary 3.13]{MalleSpathMcKay2}.  We remark that for the overlapping case that $\bG$ is type $\type{E}_6$, these maps agree. 
}

Recall that for $\la\in\irr(T)$ and $\sigma\in\gal$, we let $\delta_{\la,\sigma}\in\irr(W(\la))$ denote the linear character of $W(\la)=W(\la^\sigma)$ such that $\Lambda(\la)^\sigma=\delta_{\la,\sigma}\Lambda(\la^\sigma)$.  It will also sometimes be useful to denote $\Lambda_\la:=\Lambda(\la)$.

Let $R_\la$ be the subgroup of $N_\la$ generated by $\bg{T}^F=T$ and $\langle n_\alpha(-1)\mid \alpha\in\Phi_\la\rangle$, so that $R_\la/T\cong R(\la)$.   Note that $\Phi_\la=\Phi_{\la^\sigma}$ (see \cite[Lemma 3.2]{SFgaloisHC}). 
The following extends \cite[Lemma 3.13]{SFT20} and implies that the character $\delta_{\la,\sigma}$ of $W(\la)$ can be viewed as a character of  $C(\la)$, and hence is equal to $\delta_{\la,\sigma}'$.

\begin{lemma}\label{lem:extRla}
Keep the notation from the beginning of the section, and assume $\bG$ is not of type $\type{A}_n$. Let $\la\in\irr(T)$ and $\sigma\in\gal$.   Then $\Lambda_\la^\sigma\downarrow_{R_\la}=\Lambda_{\la^\sigma}\downarrow_{R_\la}$.  In particular, $R(\la)\leqslant \ker\delta_{\la,\sigma}$.  
\end{lemma}
\begin{proof}
Let $\la_0:=\la\downarrow_H$ and note that $\la_0^\sigma=\la_0$, since $H$ is elementary abelian.  By the construction in the proofs of \cite[Corollary 3.13]{MalleSpathMcKay2} and \cite[Lemma 4.2]{Spath09}, we see that $\Lambda_\la(x)=\Lambda_{\la^\sigma}(x)$ for $x\in V_{\la}$, and hence for $x\in V(\la):=\langle n_{\alpha}(-1)\mid \alpha\in\Phi_{\la}\rangle$. Then for $t\in T$ and $x\in V(\la)$, we have 
\[\Lambda_{\la}^\sigma(tx)=\Lambda_{\la}^\sigma(t)\Lambda_{\la}^\sigma(x)=\la^\sigma(t)\Lambda_{\la}(x)=\la^\sigma(t)\Lambda_{\la^\sigma}(x)=\Lambda_{\la^\sigma}(tx),\]
where the second equality is since $\Lambda_\la(x)\in\{\pm1\}$ for $x\in V(\la)$, thanks to \cite[Lemma 3.13]{SFT20}. 
\end{proof}

\begin{corollary}\label{cor:exceptLambda1}
Let ${G}$ be  $\type{G}_2(q)$, $\type{F}_4(q)$, $\tw{3}\type{D}_4(q)$, $\type{E}_6^\epsilon(q)$, or $\type{E}_8(q)$.  Then $\Lambda$ is $ND\times \gal$-equivariant.  Further, there is an ${N}_1D\times\gal$-equivariant extension map $\Lambda_1$ with respect to $T_1\lhd N_1$. 
\end{corollary}
\begin{proof}
The statement about $\Lambda$ is just a restatement of \prettyref{lem:exceptLambda}, except for the case $\tw{3}\type{D}_4(q)$.  In the latter case, since $C(\la)=1$ by \prettyref{lem:W(s)}, \prettyref{lem:extRla} immediately yields that the map $\Lambda$ in \prettyref{rem:trialityLambda} is also $\gal$-equivariant. Hence we assume $q\equiv 3\pmod 4$ and prove the second statement.

 If $\bG$ is not of type $\type{E}_6$, this follows analogously to  \prettyref{lem:exceptLambda} (and appealing to \prettyref{lem:extRla} in the case $\tw{3}\type{D}_4(q)$ since $C(\la)=1$), since $V_1=V, H_1=H,$ and $N_1=T_1V$ in these cases by \prettyref{lem:V1V}.  If $\bG$ is type $\type{E}_6$, then the proof of \cite[Lemma 6.1]{Spath09} show that either again $V_1=V$ and $H_1=H$ or $V_1\cong V'$ and $H_1\cong H'$, where $H'\lhd V'$ are the corresponding groups for a root system of type $\type{F}_4$.  In either case, the proof of \prettyref{lem:exceptLambda} again yields the result.
\end{proof}

\begin{corollary}\label{cor:deltaorder2}
In the situation of \prettyref{lem:extRla}, we have $\delta_{\la,\sigma}^2=1$.
\end{corollary}
\begin{proof}
If $\bG$ is of type $\type{E}_6$ or $(\zen(\bG)/\zen^\circ(\bG))^F$ is trivial or an elementary abelian $2$-group, then Lemmas \ref{lem:W(s)}, \ref{lem:exceptLambda}, and \ref{lem:extRla} imply the result, since $\delta_{\la,\sigma}$ is trivial for $\type{E}_6$ and can be thought of as a character of $C(\la)$ in each case.  

This leaves the case $\bG$ is of type $\type{D}_n$.  In this case, along the lines of the illustration in \prettyref{lem:exceptLambda}(3), the extension map $\Lambda$ is constructed using an extension map $\Lambda_0$ with respect to $H\lhd V$, which is constructed in \cite[Proposition 3.10]{MalleSpathMcKay2}.  Using an embedding of $\bg{G}$ into a group $\overline{\bg{G}}$ of type $\type{B}_n$, the map $\Lambda_0$ in \cite[Proposition 3.10]{MalleSpathMcKay2} is moreover defined in terms of an extension map $\Lambda_B$ with respect to $\overline{H}\lhd\overline{V}$ in the case of type $\type{B}$, where $\overline{V}$ is the corresponding extended Weyl group for $\overline{\bG}$ and $\overline{H}=H$ is the toral subgroup.  Specifically, we have $\Lambda_0(\la)=(\Lambda_B(\la)\downarrow_{V_\la})^t$ for a specific $t\in \bT$. Here to alleviate notation, we let $\Lambda_0(\la):=\Lambda_0(\la\downarrow_H)$ and $\Lambda_B(\la):=\Lambda_B(\la\downarrow_H)$.

  Let $\sigma\in\gal$.  Letting $\delta_B\in\irr(\overline{V}_{\la\downarrow_H}/H)$ be such that $\Lambda_B(\la)^\sigma=\delta_B\Lambda_B(\la^\sigma)$, we then have  $\Lambda_0(\la^\sigma)=(\Lambda_B(\la^\sigma)\downarrow_{V_\la})^t=((\delta_B^{-1}\Lambda_B(\la)^\sigma)\downarrow_{V_\la})^t=(\delta_B^{-1}\downarrow_{V_\la})^t(\Lambda_B(\la)^\sigma\downarrow_{V_\la})^t=(\delta_B^{-1}\downarrow_{V_\la})^t\Lambda_0(\la)^\sigma$.  Then since $\Lambda_\la(tv)=\la(t)\Lambda_0(\la)(v)$ for $t\in T$ and $v\in V_\la$, we see $\delta_{\la,\sigma}=(\delta_B)\downarrow_{V_\la}^t$.  But since $\delta_B$ has order dividing $2$, this completes the proof. 
\end{proof}

We next show that $\Lambda$ is further $\gal$-equivariant in the case of $\type{E}_7$ or $\type{B}_n$, when restricted to the members of $\irr(T)$ such that $\mathcal{E}(G, T, \la)$ contains odd-degree characters.

\begin{proposition}\label{prop:equivextB}
Let $G=\bG^F$ such that $\bG$ is simple of simply connected type $\type{B}_n$ with $n\geq3$ or $\type{E}_7$.  Assume $q$ is odd and that $\mathcal{E}(G, T,\la)\cap\irr_{2'}(G)$ is nonempty.  Then $\delta_{\la,\sigma}=1$ for any $\sigma\in\gal$.  
\end{proposition}
\begin{proof}

Note that if $\bG$ is type $\type{E}_7$, then the proof of \prettyref{prop:Clambdaeven} yields $R(\la)=W(\la)$ for any $\la\in\irr(T)$ such that $\mathcal{E}(G, T,\la)\cap\irr_{2'}(G)\neq \emptyset$, and hence $\delta_{\la,\sigma}=1$ for any $\sigma\in\gal$ in this case.

 So let $\bG$ be of type $\type{B}_n$, and assume that $R(\la)\neq W(\la)$, as otherwise we have $\delta_{\la,\sigma}=1$ trivially by appealing to \prettyref{lem:extRla}. Then $G^\ast$ is type $\type{C}_n$ and $W(\la)\cong W_{G^\ast}(s)=(\type{C}_{n/2}\times \type{C}_{n/2}).2$ with $n=2^a\geq 4$, from the proof of \cite[Lemma 4.11]{SFgaloisHC}. (See also \cite[Lemma 7.6]{MalleSpathMcKay2}.)  

Let $\Delta=\{\alpha_1,\ldots,\alpha_n\}$ be a system of simple roots for $\bG$, with labeling as in \cite[Remark 1.8.8]{gorensteinlyonssolomonIII}, so that $\alpha_i=e_i-e_{i+1}$ for $1\leq i<n$ and $\alpha_n=e_n$, where $\{e_1, \ldots, e_n\}$ is an orthonormal basis for $\R^n$.  The following can be seen using the relevant Chevalley  relations, found for example in \cite[Theorem 1.12.1]{gorensteinlyonssolomonIII}.

 Let $\zeta$ be a generator for $\F_q^\times$. Then $\lambda$ can be described by $\lambda(h_{\alpha_{n/2}}(\zeta))=-1$ and $\lambda(h_{\alpha_i}(\zeta))=1$ for $i\neq n/2$. Hence $R(\la)\cong\type{C}_{n/2}^2$ is generated by the image in $\bg{W}$ of the elements $n_{e_i}(-1)$ for $1\leq i\leq n$, together with the elements $n_{\alpha_i}(-1)$ for $1\leq i<n$ with $i\neq n/2$ and   
%
%
 $C(\la)$ is generated by the image in $\bg{W}$ of the element  
{\[c:=n_{e_1-e_{n/2+1}}(-1)n_{e_2-e_{n/2+2}}(-1)\cdots n_{e_{n/2}-e_{n}}(-1).\]}


Now, we have $h_{e_i-e_{n/2+i}}(-1)=h_{\alpha_i}(-1)\cdots h_{\alpha_{n/2+i-1}}(-1)$ for each $1\leq i\leq n/2$, so that 
\[\la(h_{e_i-e_{n/2+i}}(-1))=\la(h_{\alpha_{n/2}}(-1))=\la(h_{\alpha_{n/2}}(\zeta)^{(q-1)/2})= (-1)^{(q-1)/2}.\]  
Writing $\Lambda_\la:=\Lambda(\la)$, this yields that $\Lambda_\lambda(c)^2=\lambda(c^2)=\prod_{i=1}^{n/2}\la(h_{e_i-e_{n/2+i}}(-1))=(-1)^{n(q-1)/4}=1$.  In particular, $\Lambda_\la(c)\in\{\pm1\}$ is necessarily fixed by any $\sigma\in\gal$, so $\delta_{\la,\sigma}=1$ for all $\sigma\in\gal$ by applying \prettyref{lem:extRla}.
\end{proof}

\section{The Generic Algebra and the Character $\eta^{(\sigma)}$}\label{sec:genericalg}

Here we recall some relevant details regarding the construction of generic algebras and their specializations.  For more, we refer the reader to \cite{howlettlehrer83}, though we remark that here we work over $\overline{\mathbb{\Q}}$, rather than $\mathbb{C}$.

Keep the notation of \prettyref{sec:HCseries}. Let $(L, \la)$, with $\la\in \irr(L)$, be a cuspidal pair for $G$.   Let $\mathbf{u}:=\{u_\alpha\mid \alpha\in\Delta_\la\}$ be a set of indeterminates such that $u_\alpha=u_\beta$ if and only if $\alpha$ and $\beta$ are $W(\la)$-conjugate.  For $m\in\Q$, we will write $\mathbf{u}^m$ for the set $\{u_\alpha^m\mid\alpha\in\Delta_\la\}$ and set $\sqrt{\mathbf{u}}:=\mathbf{u}^{1/2}$. Let $\mathbb{A}_0:=\overline{\Q}[\mathbf{u}, \mathbf{u}^{-1}]$, let $\mathbb{K}$ be an algebraic closure of the quotient field of $\mathbb{A}_0$ containing $\Q(\sqrt{\mathbf{u}})$, and let $\mathbb{A}$ be the integral closure of $\mathbb{A}_0$ in $\mathbb{K}$. 

 We construct the so-called \emph{generic algebra} $\mathcal{H}$ as the free $\mathbb{A}$-module with basis $\{a_w\mid w\in W(\la)\}$, together with the unique $\mathbb{A}$-bilinear associative multiplication satisfying \cite[Theorem 4.1]{howlettlehrer83}.  Let $\mathbb{K}\mathcal{H}$ denote the $\mathbb{K}$-algebra $\mathbb{K}\otimes_{\mathbb{A}} \mathcal{H}$, which is semisimple by \cite[Corollary 4.6]{howlettlehrer83}. 
 For a ring homomorphism $h\colon \mathbb{A}\rightarrow\overline{\Q}$, we obtain the so-called \emph{specialized algebra} $\mathcal{H}^h:=\overline{\Q}\otimes_{\mathbb{A}}\mathcal{H}$ with basis $\{1\otimes a_w\mid w\in W(\la)\}$ and structure constants obtained by applying $h$ to the structure constants of $\mathcal{H}$.  If $h$ is the extension to $\mathbb{A}$ of a homomorphism $h_0\colon \mathbb{A}_0\rightarrow\overline{\Q}$ and $\mathcal{H}^h$ is semisimple, then \cite[Proposition 4.7]{howlettlehrer83} yields that there is a bijection $\varphi\mapsto \varphi^h$ between $\mathbb{K}$-characters corresponding to simple $\mathbb{K}\mathcal{H}$-modules and $\overline{\Q}$-characters of simple $\mathcal{H}^h$-modules, where $\varphi^h$ is defined by $\varphi^h(1\otimes a_w)=h(\varphi(a_w))$ for each $w\in W(\la)$.  
 
 In particular, the homomorphisms $f_0, g_0\colon \mathbb{A}_0\rightarrow\overline{\Q}$ defined by $f_0(u_\alpha)=p_{\alpha, \la}$ and $g_0(u_\alpha)=1$, respectively, yield specializations $\mathcal{H}^f$ and $\mathcal{H}^g$ satisfying $\mathcal{H}^f\cong \mathrm{End}_{\overline{\Q}G}(\HC_L^G(\la))$ and $\mathcal{H}^g\cong \overline{\Q}W(\la)$, using \cite[Lemma 4.2]{howlettlehrer83}.  (We remark that the  2-cycle in the original bijection may be taken to be trivial by \cite{geck93}.)

For the remainder of this section, let $L=T$ be a maximally split torus, so that $\mathcal{E}(G, L, \la)=\mathcal{E}(G, T, \la)$ is a principal series.  In this situation, we may follow the treatment of \cite{HowlettKilmoyer}.  In particular, let $\mathcal{H}_0$ denote the subalgebra of $\mathcal{H}$ generated by $\{a_w\mid w\in R(\la)\}$, so that $\mathcal{H}_0$ is the generic Iwahori--Hecke algebra corresponding to the Coxeter group $R(\la)$. 

We remark that since $R(\la)$ is a Weyl group, all of its characters are rational-valued, using for example \cite[Theorems 5.3.8, 5.4.5, 5.5.6, and Corollary 5.6.4]{GeckPfeiffer}.  We will use this fact throughout the remainder of the paper, and our next observations aim to give an analogue to this fact for the characters of $\mathcal{H}_0$. 

\begin{lemma}\label{lem:K0splittingfieldH0}
With the notation above, let $\mathbb{K}_0$ be $\Q(\mathbf{u})$ in the case that $R(\la)$ contains no component of type $\type{G}_2, \type{E}_7,$ or $\type{E}_8$, and let $\mathbb{K}_0:=\Q(\sqrt{\mathbf{u}})$ otherwise.  Then $\mathbb{K}_0$ is a splitting field for $\mathcal{H}_0$.  
\end{lemma}
\begin{proof}
This follows directly from \cite[Example 9.3.4 and Theorems 6.3.8 and 9.3.5]{GeckPfeiffer}.
\end{proof}

Now, using the discussion preceeding \cite[Theorem 3.7]{HowlettKilmoyer}, we may identify the group algebra $\mathbb{K} C(\la)$ with the subalgebra of $\mathbb{K}\mathcal{H}$ generated by $\{a_d\mid d\in C(\la)\}$.  In this situation, $C(\la)$ acts on $\mathbb{K}\mathcal{H}_0$ via $d(a_w)=a_{dwd^{-1}}$ for $w\in R(\la)$ and $d\in C(\la)$, and on the irreducible $\mathbb{K}$-characters $\psi_0$ of $\mathbb{K}\mathcal{H}_0$ via $\psi_0^d(a_w)=\psi_0(a_{dwd^{-1}})$.  
For an irreducible $\mathbb{K}$-character $\psi_0$ of $\mathbb{K}\mathcal{H}_0$, let $\mathbb{K}C(\la)_{\psi_0}\mathcal{H}_0$ denote the subalgebra generated by $\{a_{dw}\mid d\in C(\la)_{\psi_0}; w\in R(\la)\}$.  Then \cite[Theorem 3.7]{HowlettKilmoyer} implies that $\psi_0$ extends to an irreducible character of $\mathbb{K}C(\la)_{\psi_0}\mathcal{H}_0$. 

Further, \cite[Lemma 3.12]{HowlettKilmoyer} and the discussion preceeding it provides an analogue of Clifford theory for the characters of $\mathbb{K}\mathcal{H}$.  
Namely, if $\psi$ is an extension of $\psi_0$ as above, then any extension of $\psi_0$ to $\mathbb{K}C(\la)_{\psi_0}\mathcal{H}_0$ is of the form $\beta\psi$ for $\beta\in\irr(C(\la)_{\psi_0})$, where 
$(\beta\psi)(a_{dw}):=\beta(d)\psi(a_{dw})$ for $d\in C(\la)_{\psi_0}$ and $w\in R(\la)$.  
Further, any irreducible character of $\mathbb{K}\mathcal{H}$ is induced from such a character, where for $\psi$ a character of $\mathbb{K}C(\la)_{\psi_0}\mathcal{H}_0$, 
the induced character $\Ind_{C(\la)_{\psi_0}\mathcal{H}_0}^\mathcal{H} \psi$ is defined by 
\begin{equation}\label{eq:inducedHecke}
\Ind_{C(\la)_{\psi_0}\mathcal{H}_0}^\mathcal{H} \psi (a_w)=\frac{1}{|C(\la)_{\psi_0}|} \sum \psi(a_{dwd^{-1}})
\end{equation}
for $w\in W(\la)$, where the sum is taken over those $d\in C(\la)$ such that $a_{dwd^{-1}}\in C(\la)_{\psi_0}\mathcal{H}_0$.


\begin{lemma}\label{lem:K0splittingfieldH}
Let $\rho_0$ be an irreducible $\mathbb{K}_0$-representation of $\mathbb{K}_0\mathcal{H}_0$ affording the character $\psi_0$, with notation as in \prettyref{lem:K0splittingfieldH0}.  Then there is an extension $\rho$ of $\rho_0$ to an irreducible representation of $\mathbb{K}C(\la)_{\psi_0}\mathcal{H}_0$ 
that can be afforded over $\mathbb{K}_0$.
\end{lemma}
\begin{proof}

Let $\rho_0^{\mathbb{K}}$ denote $\rho_0$ viewed as a $\mathbb{K}$-representation. In \cite[Theorems 3.7 and 3.8]{HowlettKilmoyer}, the $\mathbb{K}$-character $\psi_0$ is shown to extend to a $\mathbb{K}$-character $\psi$ of $\mathbb{K}C(\la)_{\psi_0}\mathcal{H}_0$ 
 by constructing a representation $\rho$ extending $\rho_0^\mathbb{K}$.  In particular, $\rho(d)$ for $d\in C(\la)_{\psi_0}$ is constructed using the equivalence properties of $\rho_0^{\mathbb{K}}$.
 Since $\psi_0=\psi_0^d$ for every $d\in C(\la)_{\psi_0}$, we have $\rho_0^{\mathbb{K}}$ is equivalent over $\mathbb{K}$ to $(\rho_0^{\mathbb{K}})^d$, and hence $\rho_0$ and $\rho_0^d$ are also equivalent over $\mathbb{K}_0$, using for example \cite[Theorem 29.7]{curtisreiner62}.  But in the proof of \cite[Theorem 3.8]{HowlettKilmoyer}, this means that the construction for $\rho$ can be done over $\mathbb{K}_0$, rather than $\mathbb{K}$, yielding the statement.
\end{proof}

\begin{corollary}
Keep the notation as above. Let $\psi_0^f$ be an irreducible character of $\mathcal{H}_0^f$.  Then there is an irreducible character of $\mathcal{H}^f$ lying over $\psi_0^f$ whose values lie in $\Q(\sqrt{q})$.  If  $R(\la)$ contains no component of type $\type{G}_2, \type{E}_7, $ or $\type{E}_8$, then there is further such a character with values in $\Q$. 

Similarly, if $\psi_0^g$ is an irreducible character of $\mathcal{H}_0^g$, then there is an irreducible character of $\mathcal{H}^g$ lying over $\psi_0^g$ with values in $\Q$. 
\end{corollary}
\begin{proof}
Let $\psi_0$ be the corresponding character of $\mathbb{K}\mathcal{H}_0$ and let $\psi$ be the $\mathbb{K}_0$-character of $\mathbb{K}C(\la)_{\psi_0}\mathcal{H}_0$ extending $\psi$ guaranteed by \prettyref{lem:K0splittingfieldH}.  Then the statement follows immediately if $C(\la)_{\psi_0}=C(\la)$,  since $g(u_\alpha)=1$ and $f(u_\alpha)=p_{\alpha,\la}$ is a power of $q$ for each $\alpha\in\Delta_\la$, using \cite[Lemma 2.6]{HowlettKilmoyer}. Otherwise, we see by \eqref{eq:inducedHecke} that the character $\tau$ of $\mathcal{H}$ induced from $\psi$ will also have values in $\mathbb{K}_0$, and hence $\tau^f$ and $\tau^g$ will also take its values in the stated fields.
\end{proof}

Recall that we denote by $\mathfrak{f}$ the bijection $\mathfrak{f}\colon \irr(\mathrm{End}_{\overline{\Q}G}(\HC_T^G(\la))\rightarrow \irr(W(\la))$ induced by the specializations.  In other words, $\mathfrak{f}\colon \irr(\mathcal{H}^f)\rightarrow \irr(\mathcal{H}^g)$ is given by $\mathfrak{f}(\tau^f)=\tau^g$ for $\tau$ an irreducible character of the generic algebra $\mathcal{H}$.  Recall that for $\eta\in\irr(W(\la))$ and $\sigma\in\gal$, if $\phi\in\irr(\mathrm{End}_{\overline{\Q}G}(\HC_T^G(\la))$ is such that $\mathfrak{f}(\phi)=\eta$, then we write $\eta^{(\sigma)}$ for the character of $W(\la)$ such that $\mathfrak{f}(\phi^\sigma)=\eta^{(\sigma)}$.  Although one might hope that $\eta^{(\sigma)}=\eta^\sigma$ in general, this is known not to be the case. (For example there are situations in which every character of the Weyl group has values in $\Q$, but characters in the corresponding Iwahori--Hecke algebra take values in $\Q(\sqrt{q})$.) However, our main result of this section implies that this does hold in certain situations.

\begin{proposition}\label{prop:etasigma}
Let $\la\in\irr(T)$ and $\eta\in\irr(W(\la))$.  Let $\sigma\in\gal$, and further assume that $\sigma$ fixes $\Q(\sqrt{q})$ in the case that $R(\la)$ contains a component of type $\type{G}_2, \type{E}_7, $ or $\type{E}_8$. Then $\eta^\sigma=\eta^{(\sigma)}$.
\end{proposition}
\begin{proof}
Let $\eta\in\irr(W(\la))$ lie over $\eta_0\in\irr(R(\la))$. Note that both $\eta^\sigma$ and $\eta^{(\sigma)}$ also lie over $\eta_0$, since the characters of the Weyl group $R(\la)$ are all rational-valued.

 Write $\eta_0=\psi_0^g$ for $\psi_0$ a character of $\mathcal{H}_0$, with the notation as above, and let $\psi$ be the character of $C(\la)_{\psi_0}\mathcal{H}_0$ with values in $\mathbb{K}_0$ ensured by \prettyref{lem:K0splittingfieldH}.  Let $\tau$ be the character induced from $\psi$ to $\mathcal{H}$, so that $\tau$ also has values in $\mathbb{K}_0$ and hence $\tau^g$ and $\tau^f$ are fixed by $\sigma$.  Note that, using Gallagher's theorem and \cite[Problem (5.3)]{isaacs}, together with the fact that that $C(\la)$ is either cyclic or Klein-four, we have $\eta=(\beta \tau)^g=\beta \tau^g$ for some $\beta\in\irr(C(\la))$, since $g$ does not affect the values of $\beta$.  Then $\eta^\sigma=\beta^\sigma \tau^g=(\beta^\sigma \tau)^g$.  But also, $(\beta \tau)^f=\beta \tau^f$, and hence $\mathfrak{f}(\beta \tau^f)=\eta$, and we similarly have $(\beta \tau^f)^\sigma=\beta^\sigma \tau^f=(\beta^\sigma \tau)^f$.  Hence $\mathfrak{f}((\beta \tau^f)^\sigma)=\eta^\sigma$, so  $\eta^{(\sigma)}=\eta^\sigma$ by definition.
\end{proof}

%

\subsection{The Case $\ell=2$}

We next consider the situation of $\ell=2$.  
In the following observation, for $R\lhd H$ we denote by $\irr_{2',H}(R)\subseteq \irr_{2'}(R)$ the set of $\theta\in\irr(R)$ such that $\theta\in\irr(R|\eta)$ for some $\eta\in\irr_{2'}(H)$.

\begin{lemma}\label{lem:semidirelemab}
Let $H$ be a finite group such that $H=R\rtimes C$ with $C$ an elementary abelian $2$-group.  

\begin{enumerate}[label=(\alph*)] 
\item Let $\sigma\in\gal$ and let $\eta\in\irr_{2'}(H|\theta)$, where $\theta\in\irr(R)$ is fixed by $\sigma$.  Then $\eta^\sigma=\eta$.
\item If $C=C_2$ is cyclic, then there is a $\gal$-equivariant map $\Xi\colon \irr_{2',H}(R)\rightarrow\irr_{2'}(H)$ such that $\Xi(\theta)$ is an extension of $\theta$ to $H$ for each $\theta\in\irr_{2',H}(R)$. 
\end{enumerate}
\end{lemma}
\begin{proof}
First note that for $\eta\in\irr_{2'}(H)$, it follows from Clifford theory that $\eta$ restricts irreducibly to some $\theta\in\irr_{2'}(R)$ since $H/R$ is a $2$-group.

Now consider part (a).  Note that the statement is true for linear characters, since if $\eta$ is linear, we have $\eta^\sigma(rc)=\eta(r)^\sigma\eta(c)^\sigma=\theta(r)^\sigma\eta(c)^\sigma=\theta(r)\eta(c)=\eta(rc)$ for any $r\in R$ and $c\in C$, since $c^2=1$ and $\theta^\sigma=\theta$ by assumption.

Now, the character $\det\theta$ (obtained by composing the determinant with a representation affording $\theta$) is a linear character fixed by $\sigma$. Hence any extension $\mu$ of $\det\theta$ to $H$ is fixed by $\sigma$ by the previous paragraph.  But \cite[Lemma 6.24]{isaacs} implies that there is a unique extension $\eta'\in\irr(H)$ of $\theta$ such that $\det\eta'=\mu$, and hence $\eta'$ is fixed by $\sigma$ by uniqueness.  Then since $\eta=\eta'\beta$ for some $\beta\in\irr(C)$ by Gallagher's theorem, we see $\eta^\sigma=\eta$ as well, since $\beta$ is a character of an elementary abelian $2$-group. This completes part (a).

We now consider part (b). Note that again the statement holds for linear characters, since the extension is determined by its value on a generator $c$ of $C$. In particular, we may define $\Xi_0\colon \irr_{2',H}(R/R')\rightarrow \irr_{2'}(H/R')$ to be an extension map such that $\Xi_0(\theta)$ is an extension to $H$ satisfying $\Xi_0(\theta)(c)=1$ for each linear $\theta\in\irr_{2', H}(R)$. Then $\Xi_0(\theta)^\sigma=\Xi_0(\theta^\sigma)$ for each $\sigma\in\gal$ and each linear $\theta\in\irr_{2', H}(R)$.

Now, for general $\theta\in\irr_{2', H}(R)$, we may define $\Xi(\theta)$ to be the unique extension to $H$ such that $\det(\Xi(\theta))=\Xi_0(\det\theta)$, guaranteed by \cite[Lemma 6.24]{isaacs}.  Then since $\det(\Xi(\theta^\sigma))=\Xi_0(\det\theta^\sigma)=\Xi_0(\det\theta)^\sigma=\det(\Xi(\theta))^\sigma=\det(\Xi(\theta)^\sigma)$, this forces $\Xi(\theta)^\sigma=\Xi(\theta^\sigma)$ for each $\sigma\in\gal$ and $\theta\in\irr_{2', H}(R)$.
\end{proof}

This yields the following extensions of \cite[Proposition 4.13 and Corollary 4.16]{SFgaloisHC}.

\begin{corollary}\label{cor:newsigmaaction}
Let ${G}$ be a group of Lie type with odd defining characteristic and let $\lambda\in\irr(T)$ such that $C(\lambda)$ is a (possibly trivial) elementary abelian $2$-group. 
Let $\sigma\in\gal$ and $\eta\in\irr_{2'}(W(\lambda))$.  
Then $\eta^{(\sigma)}=\eta=\eta^\sigma$.
\end{corollary}

\begin{proof}
Since $R(\la)$ is a Weyl group, every $\theta\in\irr(R(\la))$ is rational-valued, so the statement $\eta=\eta^\sigma$ follows directly from \prettyref{lem:semidirelemab}, since $W(\la)=R(\la)\rtimes C(\la)$.
If $R(\la)$ contains no component $\type{G}_2, \type{E}_7,$ or $\type{E}_8$, \prettyref{prop:etasigma} yields $\eta^\sigma=\eta^{(\sigma)}$.  However, note that the exceptions to $\mathbb{K}_0=\mathbb{Q}(\mathbf{u})$ in  \prettyref{lem:K0splittingfieldH0} in case $\type{G}_2$, $\type{E}_7$, and $\type{E}_8$ have even degree, using \cite[Theorem 8.3.1, Examples 9.3.2 and 9.3.4]{GeckPfeiffer}, and hence we may remove the assumption that $\sigma$ fixes $\Q(\sqrt{q})$ in this case in \prettyref{prop:etasigma}.  
\end{proof}


\section{The Bijection for the Principal Series}\label{sec:principalseries}

In this section, we prove part \ref{mainodd} of \prettyref{thm:main}, with the 
exception of $\Sp_{2n}(q)$ when $\ell= 2$ and $q\equiv 5\pmod 8$.  This exception will be completed in \prettyref{sec:typeC}.

As before, let $\bG$ be simple of simply connected type and let $F\colon \bG\rightarrow \bG$ be a Frobenius morphism defining $G=\bG^F$ over $\F_q$, and $G\lhd \wt{G}$ as a result of a regular embedding.

We keep the notation from \prettyref{sec:prelim}. In particular, we let $T=\bg{T}^F$ be a maximally split torus of $G$ and let $N=\norm_{\bg{G}}(\bg{T})^F=\norm_{{G}}(\bg{S}_1)$ and $\wt{N}=\norm_{\wt{{G}}}(\bg{S}_1)$ where $\bg{S}_1$ is a torus of $\bg{G}$ such that $\cen_{G}(\bg{S}_1)=T$.  Further, $D\leqslant \aut(G)$ denotes the group of automorphisms generated by graph and field automorphisms.  For $\la\in\irr(T)$, recall that $\mathcal{E}(G, T,\la)$ denotes the set of irreducible constituents of $\HC_T^G(\la)$ and that $\Lambda$ is the $ND$-equivariant extension map with respect to $T\lhd N$ discussed in \prettyref{sec:extnmap}. 

According to \cite[Theorem 5.2]{MalleSpathMcKay2}, the map 
\[\Omega\colon\bigcup_{\la\in\irr(T)} \mathcal{E}(G, T,\la)\rightarrow \irr(N) \] given by 
\begin{equation}\label{eq:bij}
\HC_T^G(\la)_\eta \mapsto \Ind_{N_{\la}}^N(\Lambda(\la)\eta)
\end{equation} defines an $\wt{N}D$-equivariant bijection (see also \cite[Theorem 4.5]{CS13} and its proof when $\zen(G)=1$).  

 Using the results of the previous sections, we may show that $\Omega$ is equivariant with respect to the Galois automorphisms in $\galh_\ell$ when $\ell$ is an odd prime dividing $q-1$.
 
 \begin{proposition}\label{prop:equivariancemaxsplit}
 Let $\Omega$ be as in \eqref{eq:bij} and assume $G=\bG^F$ is as above, not of type $\type{A}$.  Let $\ell$ be an odd prime dividing $(q-1)$.  Then $\Omega(\chi^\sigma)=\Omega(\chi)^\sigma$ for any $\chi\in \bigcup_{\la\in\irr(T)} \mathcal{E}(G, T,\la)$ and $\sigma\in \galh_\ell$.  
 \end{proposition}
 \begin{proof}
 
Let $\chi=\HC_T^G(\la)_\eta$.  Then by \prettyref{thm:GaloisAct} and the definition \eqref{eq:bij} of $\Omega$, we have 
\[\Omega(\chi^\sigma)=\Omega(\HC_T^G(\la^\sigma)_{\gamma_{\la,\sigma}(\delta'_{\la,\sigma})^{-1}\eta^{(\sigma)}})=\Ind_{N_\la}^N\left(\Lambda(\la^\sigma)\gamma_{\la,\sigma}(\delta'_{\la,\sigma})^{-1}\eta^{(\sigma)}\right).\]
  But $\eta^{(\sigma)}=\eta^\sigma$ and  $\gamma_{\la,\sigma}=1$, by \prettyref{prop:etasigma} and \prettyref{cor:maxsplitsqrt}.  Further, $\delta_{\la,\sigma}=\delta_{\la,\sigma}'$ by \prettyref{lem:extRla}. Hence, we have 
  \[\Omega(\chi^\sigma)=\Ind_{N_\la}^N\left(\Lambda(\la^\sigma)\delta_{\la,\sigma}^{-1}\eta^{\sigma}\right).\]

On the other hand, 
\[\Omega(\chi)^\sigma=\left(\Ind_{N_{\la}}^N(\Lambda(\la)\eta)\right)^\sigma=\Ind_{N_{\la}}^N(\Lambda(\la)^\sigma\eta^\sigma)=\Ind_{N_{\la}}^N(\Lambda(\la^\sigma)\delta_{\la,\sigma}\eta^\sigma).\]  The result now follows, since $\delta_{\la,\sigma}=\delta_{\la,\sigma}^{-1}$ by \prettyref{cor:deltaorder2}.
 \end{proof}

 By \cite[Theorems 5.14 and 5.19]{malleheightzero}, when $d_\ell(q)=1$, we have $N$ contains $\norm_G(Q)$ for some Sylow $\ell$-subgroup $Q$ (and hence $N$ is also $\aut(G)_Q$-stable by  \cite[Proposition 2.5]{CS13}), 
with the following exceptions: $\ell=3$ and $G=\SL_3(q)$ with $q\equiv  4, 7\pmod 9$, $\ell=3$ and $G=G_2(q)$ with $q\equiv 4,7\pmod 9$, or $\ell=2$ and $G=\Sp_{2n}(q)$ with $q\equiv 5\pmod 8$.  
  We now complete the proof of \prettyref{thm:main}\ref{maind=1}, with these exceptions.

\begin{theorem}\label{thm:mainsplitinline}
Let $q$ be a power of a prime $p$ and let $\bG$ be simple of simply connected type, not of type $\type{A}$.  Let $F\colon \bG\rightarrow\bG$ be a Frobenius morphism such that $G:=\bG^F$ is defined over $\F_q$ and let $\ell$ be a prime such that $d_\ell(q)=1$.  Further assume that $q\equiv 1 \pmod 9$ if $\ell=3$ and $\bG$ is type $\type{G}_2$, and $q\equiv 1\pmod 8$ if $\ell=2$ and $\bG$ is of type $\type{C}$.    Then the map defined by \prettyref{eq:bij} induces an $\wt{N}D\times \galh_\ell$-equivariant bijection 
\[\Omega_\ell\colon \irr_{\ell'}(G)\rightarrow\irr_{\ell'}(N).\]
In particular, \prettyref{cond:condition} holds, taking $M=N$.  
\end{theorem}
\begin{proof}

From \cite[Proposition 7.3]{malleheightzero}, every $\chi\in\irr_{\ell'}(G)$ lies in a Harish-Chandra series $\mathcal{E}(G, T,\la)$ for $T$ a maximally split torus of $G$ and by \cite[Proposition 7.8]{malleheightzero}, the bijection defined by \eqref{eq:bij} indeed yields a bijection \[\Omega_\ell\colon \irr_{\ell'}(G)\rightarrow\irr_{\ell'}(N)\]
such that corresponding characters lie over the same character of $\zen(G)$. 
 By \cite[Theorems 5.14 and 5.19]{malleheightzero} and \cite[Proposition 2.5]{CS13}, $N$ is $\aut(G)_Q$-stable and contains $\norm_G(Q)$ for some Sylow $\ell$-subgroup $Q$.  If $\ell$ is odd, the result now follows from \prettyref{prop:equivariancemaxsplit}.

Hence we assume $\ell=2$, $q\equiv 1\pmod 4$, and further $q\equiv 1\pmod 8$ if $G$ is $\Sp_{2n}(q)$.  Let $\sigma\in\galh_2$.  We must show that $\Omega(\chi)^\sigma=\Omega(\chi^\sigma)$ for each $\chi\in\irr_{2'}(G)$.  Note that \prettyref{lem:extRla} and \prettyref{cor:deltaorder2} imply that $\delta_{\la,\sigma}'=\delta_{\la,\sigma}=\delta^{-1}_{\la, \sigma}$.

 By \prettyref{lem:MS7.9}, we have $\chi=\HC_T^G(\la)_\eta$ for some $\la\in\irr(T)$ and $\eta\in\irr_{2'}(W(\la))$. Then we have $\gamma_{\la, \sigma}=1$, using \cite[Lemma 4.11]{SFT20} for $\Sp_{2n}(q)$ in the case being considered,  and  \prettyref{prop:Clambdaeven} otherwise.  
Further, we have $\eta^\sigma=\eta^{(\sigma)}$ by \prettyref{prop:etasigma} or \prettyref{cor:newsigmaaction}.
Then with this, the statement follows by the same calculation as the proof of \prettyref{prop:equivariancemaxsplit}.
\end{proof}




This proves \prettyref{thm:main}\ref{mainodd} with the exception of the case $G=\Sp_{2n}(q)$, $\ell=2$, and $q\equiv 5\pmod 8$, which we consider in the next section.

\section{$\Sp_{2n}(q)$ with $\ell=2$}\label{sec:typeC}

In this section, we prove \prettyref{thm:main} in the case $\ell=2$ and $G=\Sp_{2n}(q)$. Throughout this section, let $\galh:=\galh_2$ and let $G=\Sp_{2n}(q)$ with $q$ odd.  Since the case $q\equiv1\pmod8$ is completed in \prettyref{thm:mainsplitinline} above, we assume throughout that $q$ is an odd power of an odd prime.   In \cite[Section 4.4]{Malle08},  Malle shows that $G$ satisfies the inductive McKay conditions for $\ell=2$ in this case.  In particular, he constructs a proper $\aut(G)_Q$-stable subgroup $M\leqslant G$ containing $\norm_G(Q)$, where $Q$ is a Sylow $2$-subgroup of $G$, such that there is an $\aut(G)_Q$-equivariant bijection $\irr_{2'}(G)\leftrightarrow\irr_{2'}(M)$ satisfying that corresponding characters lie over the same character of $\zen(G)$, in addition to the stronger properties required in the inductive McKay conditions. We will show that this bijection can be chosen to further be equivariant with respect to $\galh$.  

The following, found in \cite[Corollary 14.3]{SFT20}, describes the character $\gamma_{\la,\sigma}\delta_{\la,\sigma}$ in the case of $\Sp_{2n}(q)$, and will be needed to show the $\galh$-equivariance of the bijection.

\begin{lemma}[Corollary 14.3 of \cite{SFT20}]\label{lem:rdelta2}
Let $G=\Sp_{2n}(q)$ with $q$ odd, and let $\la\in\irr(T)$ be a nontrivial character of a maximally split torus $T$ of $G$ such that $\la^2=1$.  Let $\ell=2$ and $\sigma\in\galh_2$.  Then $\gamma_{\la,\sigma}\delta_{\la,\sigma}\in\irr(C(\la))\cong\{\pm1\}$ satisfies:
\begin{itemize}
\item If $q\equiv \pm1\pmod 8$, then $\gamma_{\la,\sigma}\delta_{\la,\sigma}=1$.
\item If $q\equiv \pm 3\pmod 8$, then $\gamma_{\la,\sigma}\delta_{\la,\sigma}=(-1)^r$, where $r$ is the integer such that $\sigma$ sends each $\ell'$-root of unity to the $\ell^r$ power.
\end{itemize}
In particular, $\gamma_{\la,\sigma}\delta_{\la,\sigma}$  is nontrivial if and only if $\sqrt{\omega q}$ is moved by $\sigma$, where $\omega=(-1)^{(q-1)/2}$.

\end{lemma}

We begin by recalling the group $M$ and bijection from \cite[Section 4.4]{Malle08}, which depends on whether or not $n$ is a power of $2$.  Write $n=\sum_{j\in{J}}2^j$ for the $2$-adic decomposition of $n$.  

\subsection{The Case $n=2^j$}

Here suppose that $|J|=1,$ so that $n$ is a power of $2$.  In this case, the group $M$ from \cite[Section 4.4]{Malle08} is a wreath product $\Sp_n(q)\wr 2$.  Write $X$ for the base group $\Sp_{n}(q)\times \Sp_n(q)$, embedded naturally as $2$-by-$2$ block matrices.  The members of $\irr_{2'}(M)$ are the extensions to $M$ of the characters $\mu\otimes\mu$ of $X$, where $\mu\in\irr_{2'}(\Sp_n(q))$.  For $\mu\in\irr_{2'}(\Sp_n(q))$, we let $\Xi(\mu)$ denote the corresponding extension of $\mu\otimes\mu$ to $M$ such that $\Xi(\mu^\sigma)=\Xi(\mu)^\sigma$ for all $\sigma\in\gal$, guaranteed by \prettyref{lem:semidirelemab}(b).  Note that by Gallagher's theorem, the other extension of $\mu\otimes\mu$ is $\Xi(\mu)\beta$, where $\beta\in\irr(M/X)$ has order 2, and $\Xi(\mu^\sigma)\beta=\Xi(\mu)^\sigma\beta=(\Xi(\mu)\beta)^\sigma$ since $\beta$ is necessarily fixed by $\gal$.

\begin{lemma}\label{lem:Cnpower2}
Let $G=\Sp_{2n}(q)$ for $q$ an odd power of an odd prime and $n\geq 2$ a power of $2$, and let $\ell=2$.  Then \prettyref{cond:condition} holds for $G$, taking $M=\Sp_n(q)\wr 2$ as above.  
\end{lemma} 
\begin{proof}
Keeping the notation as above, the bijection $\irr_{2'}(G)\rightarrow\irr_{2'}(M)$ is constructed in \cite[Theorem 4.10]{Malle08} as any bijection sending odd-degree unipotent characters of $G$ to the characters of the form $\Xi(\mu)$ and $\Xi(\mu)\beta$ for $\mu$ an odd-degree unipotent character of $\Sp_n(q)$, and sending non-unipotent characters to those of the form $\Xi(\mu)$ and $\Xi(\mu)\beta$ for $\mu$ non-unipotent.  From \cite[Corollary 1.12]{lus02}, we know the unipotent characters of $G$ and $X$ are rational-valued, and hence fixed by every $\sigma\in\gal$.  Then the same is true for the characters of $\irr_{2'}(M)$ lying above unipotent characters of $X$, by \prettyref{lem:semidirelemab}. In particular, the bijection is $\galh$-equivariant on unipotent characters.     

Using \cite[Lemma 4.1 and Proposition 4.6]{Malle08} and \prettyref{thm:MS7.7}, we see that for any positive integer $k$, the non-unipotent characters of $\Sp_{2^k}(q)$ of odd degree lie in a single Harish-Chandra series, namely a principal series $\mathcal{E}(\Sp_{2^k}(q), T, \la)$ with $\la^2=1$.  
By \prettyref{prop:etasigma}, $\eta=\eta^{(\sigma)}=\eta^\sigma$ for all $\eta\in\irr(W(\la))$ here.  Hence the action of $\galh$ on $\mathcal{E}(\Sp_{2^k}(q), T, \la)$ is determined by the character $\gamma_{\la,\sigma}\delta_{\la,\sigma}$, by \prettyref{thm:GaloisAct}. 

From \prettyref{lem:rdelta2}, we note that $\gamma_{\la,\sigma}\delta_{\la,\sigma}$ is independent of $k$, so that $\galh$ permutes pairs of non-unipotent members of $\irr_{2'}(G)$ exactly when it permutes pairs of non-unipotent members of $\irr_{2'}(X)$ (namely, when $\sqrt{\omega q}^\sigma\neq\sqrt{\omega q}$), and in an analogous way.  Then the bijection $\irr_{2'}(G)\rightarrow\irr_{2'}(M)$ may be chosen to be $\galh$-equivariant by ensuring that each pair $\{\HC_T^G(\la)_\eta, \HC_T^G(\la)_{(-1_{C(\la)})\eta}\}$ is mapped to pairs of the form \[\left\{\Xi\left(\HC_{T'}^{\Sp_n(q)}(\la')_{\eta'}\right), \Xi\left(\HC_{T'}^{\Sp_n(q)}(\la')_{(-1_{C(\la_1)})\eta'}\right)\right\}\] or of the form \[\left\{\Xi\left(\HC_{T'}^{\Sp_n(q)}(\la')_{\eta'}\right)\beta, \Xi\left(\HC_{T'}^{\Sp_n(q)}(\la')_{(-1_{C(\la_1)})\eta'}\right)\beta\right\}\] for $\eta\in\irr_{2'}(W(\la))$ and $\eta'\in\irr_{2'}(W(\la'))$, where $\mathcal{E}(G, T,\la)$ and $\mathcal{E}(\Sp_n(q), T',\la')$ are the unique Harish-Chandra series of $G$ and $\Sp_n(q)$, respectively, containing non-unipotent odd-degree characters. This proves the statement.
%
%
%
%
\end{proof}

\subsection{The Case $n\neq 2^j$}

When $|J|\geq 2$, let $m=\max_{j\in J}\{2^j\}$.  In this case, Malle defines the subgroup $M$ in \cite[Theorem 4.11]{Malle08} as $\Sp_{2(n-m)}(q)\times \Sp_{2m}(q)$, naturally embedded as block matrices.  Writing $M_1:=\Sp_{2(n-m)}(q)$ and $M_2:=\Sp_{2m}(q)$, we see using \cite[Propositions 4.2, 4.5]{Malle08} that the members of $\irr_{2'}(X)$ for any $X\in\{G, M_1, M_2\}$ are found in Lusztig series $\mathcal{E}(X, t)$ indexed by semisimple classes corresponding to the identity $t=1$ and classes of $2$-central involutions $t$ of the dual group $X^\ast$,  and that these semisimple classes are then in bijection with subsets of $J$, $J\setminus \{m\}$, and $\{m\}$, respectively. Here the empty subset corresponds to the series of unipotent characters.  Write $\mathcal{E}_{2'}(X, t)$ for the set $\mathcal{E}(X, t)\cap \irr_{2'}(X)$.

\begin{proposition}
Let $G=\Sp_{2n}(q)$ for $q$ an odd power of an odd prime and $n\geq 2$, and let $\ell=2$.  Then \prettyref{cond:condition} holds for $G$, taking $M$ to be the subgroup of $G$ described above.  
\end{proposition} 
\begin{proof}
By \prettyref{lem:Cnpower2}, we may assume $n$ is not a power of $2$ and keep the notation above.  Let $s_1\in M_1^\ast$ and $s_2\in M_2^\ast$ correspond to subsets $I_1$ and $I_2$ of $J\setminus\{m\}$ and $\{m\}$, respectively.  Then the bijection from \cite[Theorem 4.11]{Malle08} sends $\mathcal{E}_{2'}(M_1, s_1)\otimes \mathcal{E}_{2'}(M_2, s_2)\subset \irr_{2'}(M)$ to the set $\mathcal{E}_{2'}(G,s)\subset\irr_{2'}(G)$, where $s$ corresponds to the subset $I_1\cup I_2$ of $J$.  This bijection naturally sends products of unipotent characters to a unipotent character, and hence is $\galh$-equivariant on unipotent characters since they are again rational-valued.

Let $\chi\in\irr_{2'}(X)$ be a non-unipotent character in a series parametrized by a semisimple element corresponding to the subset $I$, where $X:=\Sp_{2k}(q)\in\{G, M_1, M_2\}$.  By \prettyref{thm:MS7.7},  either $\chi\in\mathcal{E}(X, T,\la)$ where $T$ is a maximally split torus of $X$ and $\la\in\irr(T)$ satisfies $\la^2=1$, or $q\equiv3\pmod4$, $0\in I$, and $\chi\in\mathcal{E}(X, L,\la)$, where $L\cong \Sp_2(q)\times T_0$ with $T_0$ a maximally split torus of $\Sp_{2(k-1)}(q)$.  In the latter case, we further have $\la=\psi\otimes \la_0$, where $\la_0\in\irr(T_0)$ satisfies $\la_0^2=1$ and $\psi$ is one of the two characters $\psi_1$ or $\psi_2$ of $\Sp_2(q)$ of degree $\frac{q-1}{2}$. 

Let $\omega\in\{\pm1\}$ be such that $q\equiv \omega \pmod 4$. Note that $\sigma\in\galh$ interchanges $\psi_1$ and $\psi_2$ if and only if $\sqrt{\omega q}^\sigma\neq\sqrt{\omega q}$.  Further, \cite[Theorem B]{SFT20} and its proof yields that more generally, the non-unipotent members of $\irr_{2'}(X)$ are fixed if $\sqrt{\omega q}^\sigma=\sqrt{\omega q}$ and are permuted by $\sigma$ analogously in pairs via \[\left(\HC_T^G(\la)_\eta\right)^\sigma=\HC_T^G(\la)_{(-1_{C(\la)})\eta} \quad\text{ and }\quad \left(\HC_L^G(\psi_1\otimes \la_0)_\eta\right)^\sigma=\HC_L^G(\psi_2\otimes \la_0)_{{(-1_{C(\la_0)})\eta}},\] if $\sqrt{\omega q}^\sigma\neq\sqrt{\omega q}$.  

Recall that non-unipotent characters in $\irr_{2'}(G)$ are mapped to $\chi_1\otimes\chi_2 \in\irr_{2'}(M_1\times M_2)$ such that at least one of $\chi_1$ or $\chi_2$ is also non-unipotent.
Then this yields that the bijection $\irr_{2'}(G)\rightarrow\irr_{2}(M)$ may be fixed to be $\galh$-equivariant by sending pairs $\{\HC_T^G(\la)_\eta, \HC_T^G(\la)_{(-1_{C(\la)})\eta}\}$ or $\{\HC_L^G(\psi_1\otimes \la_0)_\eta, \HC_L^G(\psi_2\otimes \la_0)_{{(-1_{C(\la_0)})\eta}}\}$ to pairs in $\irr_{2'}(M)$ whose non-unipotent components are analogous pairs in $\irr_{2'}(M_1)$ and $\irr_{2'}(M_2)$ and whose unipotent components (if applicable) are equal.  Then the characters in these pairs are necessarily interchanged by $\sigma\in\galh$ if and only if the same is true for the pairs in $\irr_{2'}(G)$ mapped to them.
%
\end{proof}

This completes the proof of \prettyref{thm:main}\ref{maind=1}, as well as the case of $G=\Sp_{2n}(q)$ in \prettyref{thm:main}\ref{main2}.

\section{Remaining Cases for $\ell=2$ and $d_2(q)=2$}\label{sec:typeB}

In this section, we complete the proof of \prettyref{thm:main}\ref{main2}.  
In particular, we will consider the cases $G=\type{G}_2(q)$, $\tw{3}\type{D}_4(q)$, $\type{F}_4(q),$ $\type{E}_6^\epsilon(q)_{sc}$, $\type{E}_7(q)_{sc}$, $\type{E}_8(q)$, or
 $\type{B}_n(q)_{sc}$ with $n\geq3$, especially when $q\equiv 3\pmod 4$.  Here we write $\type{E}_6^\epsilon$ for $\epsilon\in\{\pm1\}$ to denote $\type{E}_6$ in the case $\epsilon=+1$ and $\tw{2}\type{E}_6$ in the case $\epsilon=-1$.  

 Recall that we have $\bg{T}=\cen_{\bg{G}}(\bg{S})$ for some Sylow $2$-torus $\bg{S}$ of $(\bg{G}, vF)$ and $N_1:=\norm_{\bg{G}}(\bg{S})^{vF}=\bg{N}^{vF}$.
 By \cite[Theorem 7.8]{malleheightzero}, if $q\equiv 3\pmod 4$, then $N_1$ contains $\norm_G(Q)$ for a Sylow $2$-subgroup $Q$ of $G$ (and hence $N_1$ is also $\aut(G)_Q$-stable by  \cite[Proposition 2.5]{CS13}), and there is a bijection 
\[ \Omega_1\colon \irr_{2'}(G)\rightarrow\irr_{2'}(N_1),\] which is moreover $\aut(G)_Q$-equivariant by \cite[Proposition 4.5]{CS13} and \cite[Theorem 6.3]{MalleSpathMcKay2} combined with \cite[Theorem 2.12]{spath12}, and where corresponding characters lie over the same character of $\zen(G)$.  
 Hence it suffices to show that this bijection can be chosen to further be $\galh_2$-equivariant.  In fact, for the listed groups aside from $\type{E}_6^\epsilon(q)$, we will show that every member of $\irr_{2'}(G)$ and of $\irr_{2'}(N_1)$ is rational-valued.

\begin{proposition}\label{prop:d2rational}
Let $q$ be odd and let $G=\type{G}_2(q)$, $\tw{3}\type{D}_4(q)$, $\type{F}_4(q),$ $\type{E}_7(q)_{sc}$, $\type{E}_8(q)$, or
 $\type{B}_n(q)_{sc}$ with $n\geq 3$.  Then every member of $\irr_{2'}(G)$ is rational-valued.  Further, keeping the notation above, we have every member $\irr_{2'}(N_1)$ is rational-valued when $q\equiv 3\pmod 4$.
 \end{proposition}
 
Before we prove this statement, we introduce a little more notation for the situation $q\equiv 3\pmod 4$.  Let $(\bG^\ast, (vF)^\ast)$ be dual to $(\bG, vF)$.  By \cite[Lemma 3.3]{malleheightzero}, the torus $\bT^\ast$ dual to $\bT$ can be identified with $\cen_{\bG^\ast}(\bg{S}^\ast)$ and we write $\bg{N}^\ast:=\norm_{\bG^\ast}(\bg{S}^\ast)$,  $N_1^\ast:=(\bg{N}^\ast)^{(vF)^\ast}$, and $T_1^\ast:=(\bT^\ast)^{(vF)^\ast}$.

By the proof of \cite[Theorem 7.8]{malleheightzero}, we see that both $\irr_{2'}(\bG^{vF})$ (and hence $\irr_{2'}(G)$ since $G\cong \bG^{vF}$) and $\irr_{2'}(N_1)$ are in bijection with pairs $(\lambda_1, \eta_1)$ or $(s, \eta_1)$, where $\lambda_1\in\irr(T_1)$ satisfies $[N_1:(N_1)_{\la_1}]$ is odd and  $\la_1\in\mathcal{E}(T_1, s)$ with $s\in T_1^\ast$ a semisimple element centralizing a Sylow $2$-subgroup of $N_1^\ast$ (and hence of ${\bG^\ast}^{(vF)^\ast}$), up to $N_1^\ast$-conjugation; and $\eta_1\in\irr_{2'}(W_1(\la_1))$.  Here we define $W_1(\la_1):=(N_1)_{\la_1}/T_1$, which by \cite[Proposition 7.7]{malleheightzero} is isomorphic to $W_{N_1^\ast}(s):=\cen_{N_1^\ast}(s)/T_1^\ast$.

Now, the member of $\irr_{2'}(N_1)$ corresponding to $(\la_1, \eta_1)$ is of the form $\Ind_{(N_1)_{\la_1}}^{N_1}(\Lambda_1(\la_1)\eta_1)$, where $\Lambda_1$ is an extension map with respect to $T_1\lhd N_1$. Note that since the Weyl group of $\bG$ and of $\bg{N}$ are the same and $N_1=\bg{N}^{vF}$, we have $W_{N_1^\ast}(s)$ can be thought of as the fixed points under $(vF)^\ast$ of the Weyl group of the possibly disconnected group $\cen_{\bG^\ast}(s)$, which in turn is the set of $w\in (\bg{W}^\ast)^{(vF)^\ast}$ such that $s^w=s$.    Write $W_1(s)$ for this group, so that $W_1(s)\cong W_1(\la_1)$.  Further, write $W_1^\circ(s)$ for the fixed points under $(vF)^\ast$ of the Weyl group of the connected component $\cen_{\bG^\ast}(s)^\circ$.  The group $W_1^\circ(s)$ is then a true Weyl group (in the sense of $(B,N)$ pairs and Coxeter groups) and is the subgroup of $W_1(s)$ generated by the simple reflections corresponding to $\alpha^\ast\in\Phi^\ast$ such that $\alpha^\ast(s)=1$, where we write $\Phi^\ast$ for the root system of $\bG^\ast$. (See, e.g. \cite[Remark 2.4]{dignemichel}.)  Then by the isomorphism $\irr(T_1)\cong T_1^\ast$ given by \cite[Propositions 4.2.3, 4.4.1]{carter2}, it follows that 
\begin{equation}\label{eq:W_1R_1}
W_1^\circ(s)\cong  R_1(\la_1)\quad\text{ and }\quad W_1(s)/W_1^\circ(s)\cong W_1(\la_1)/R_1(\la_1),
\end{equation}
 where we define $R_1(\la_1)$ to be the reflection group generated by the simple reflections $s_\alpha$ for $\alpha\in\Phi$ such that $s_\alpha\in W_1(\la_1)$ and $T_1\cap X_{\pm\alpha}$ is in the kernel of $\la_1$.

 \begin{proof}[Proof of \prettyref{prop:d2rational}]
 
 By \prettyref{thm:MS7.7} and \prettyref{lem:MS7.9}, we know every member of $\irr_{2'}(G)$ is of the form $\chi=\HC_T^G(\la)_\eta\in\mathcal{E}(G, T,\la)$, where $[W:W(\la)]$ and $\eta(1)$ are both odd.  
 Note that using \prettyref{lem:sinvol}, we see $s^2=1$, and it follows that $\la^2=1$ and hence $\la^\sigma=\la$. 
Now, \prettyref{prop:Clambdaeven} yields that $\gamma_{\la,\sigma}$ is trivial for any $\sigma\in\gal$.  Further, $\eta=\eta^{(\sigma)}=\eta^\sigma$ for any $\eta\in\irr_{2'}(W(\la))$ and $\sigma\in\gal$, using 
\prettyref{cor:newsigmaaction}.  Further, \prettyref{cor:exceptLambda1}  and \prettyref{prop:equivextB}  yields that $\delta_{\la,\sigma}=1$ for any $\sigma\in\gal$. From this, we now see that $\chi^\sigma=\chi$ for each $\chi\in\irr_{2'}(G)$ and $\sigma\in\gal$, using \prettyref{thm:GaloisAct}, and we turn our attention to $\irr_{2'}(N_1)$.


 

Let $q\equiv 3\pmod 4$. As in the preceeding discussion, $\chi\in\irr_{2'}(G)$ and $\Omega_1(\chi)$ are also parametrized by a pair $(\la_1, \eta_1)$ with the properties described above. Let $\sigma\in\gal$ and write $\Ind_{(N_1)_{\la_1}}^{N_1}(\Lambda_1(\la_1)\eta_1)$ for $\Omega_1(\chi)$, so that $\Omega_1(\chi)^\sigma=\Ind_{(N_1)_{\la_1}}^{N_1}(\Lambda_1(\la_1)^\sigma\eta_1^\sigma)$.  Note that \prettyref{lem:semidirelemab} gives $\eta_1^\sigma=\eta_1$, since $R_1(\la_1)$ is a Weyl group by \eqref{eq:W_1R_1} and hence has rational-valued characters.  
So, noting that $\la_1^\sigma=\la_1$ since $s^2=1$, it now suffices to show $\delta=1$, where $\delta$ is the character of $W_1(\la_1)$ such that $\Lambda_1(\la_1)^\sigma=\Lambda_1(\la_1^\sigma)\delta$ guaranteed by Gallagher's theorem.  In the case $G=\type{G}_2(q)$, $\tw{3}\type{D}_4(q)$, $\type{F}_4(q),$ or $\type{E}_8(q)$, this is accomplished by \prettyref{cor:exceptLambda1}, completing the proof in these cases.  

We may therefore assume $G$ is type $\type{E}_7$ or $\type{B}_n$.  Now, $\delta$ is trivial on $R_1(\la_1)$ by the same proof as \cite[Lemma 3.13]{SFT20}.  Then taking into consideration \eqref{eq:W_1R_1}, we may argue analogously to \prettyref{prop:equivextB}, replacing $\zeta$ with a generator of the cyclic group of size $q+1$ in $\F_{q^2}^\times$ in the case $\type{B}_n$, to obtain $\delta=1$.
%
%
\end{proof}
 

 \begin{proposition}\label{prop:E6}
 Let $q\equiv 3\pmod 4$ and let $G=\type{E}_6^\epsilon(q)_{sc}$ and keep the notation from before.  Then the  $\aut(G)_Q$-equivariant bijection $\Omega_1\colon\irr_{2'}(G) \rightarrow \irr_{2'}(N_1)$ is also $\galh_2$-equivariant.  
 \end{proposition}
 \begin{proof}
 Throughout, we identify $G$ with $\bG^{vF}$.
By analyzing the possible centralizer structures $\cen_{G^\ast}(s)$ of semisimple elements of $G^\ast$ (see also \cite[Lemma 4.13]{navarro-tiep:2015:irreducible-representations-of-odd-degree}), we see that the only possibilities for $s$ yielding $\chi\in\mathcal{E}(G,s)$ with odd degree are $s=1$ and those with $\cen_{G^\ast}(s)$ of type $\type{D}_5^\epsilon(q) \times (q-\epsilon)$.  Recall that odd-degree characters in $\mathcal{E}(G,s)$ are in bijection with odd-degree unipotent characters of $\cen_{G^\ast}(s)$ in such a way that if $\chi$ corresponds to the unipotent character $\psi$, we have $\chi(1)=[G^\ast:\cen_{G^\ast}(s)]_{p'}\psi(1)$.  

The unipotent characters of odd degree of $G$ are all fixed by $\gal$, using \cite[Proposition 4.4]{SFgaloisHC}. In the non-unipotent cases, $\cen_{\bG^\ast}(s)$ is connected, and the eight odd-degree unipotent characters of $\cen_{G^\ast}(s)$ have distinct degrees. Hence, the image of $\chi\in\mathcal{E}(G,s)\cap \irr_{2'}(G)$ under $\sigma\in\galh_2$ is completely determined by the image $\mathcal{E}(G,s)^\sigma$ of $\mathcal{E}(G,s)$ under $\sigma$. Further, \cite[Lemma 3.4]{SFT18a} implies that $\mathcal{E}(G,s)^\sigma=\mathcal{E}(G,s^\sigma)$, where if $\sigma\in\galh_2$ sends odd roots of unity to the $2^r$ power and $2$-power roots of unity to the power $b$, we define $s^\sigma:=s_{2'}^{2^r}s_{2}^b$, where $s=s_{2'}s_2=s_2s_{2'}$ with $|s_{2'}|$ odd and $|s_2|$ a power of $2$.

Then if $\chi\in\irr_{2'}(G)$ corresponds to the pair $(s, \eta_1)$ or $(\la_1, \eta_1)$, we see that for $\sigma\in\galh_2$, $\chi^\sigma$ corresponds to the pair $(s^\sigma, \eta_1)=(s^\sigma, \eta_1^\sigma)$.  (Note that since $\cen_{\bG^\ast}(s)$ is connected, $W_1(\la_1)\cong W_1(s)$ is a Weyl group in this case, and hence $\eta_1$ is rational-valued.)

On the other hand, let $\Omega_1(\chi)=\Ind_{(N_1)_{\la_1}}^{N_1}(\Lambda_1(\la_1)\eta_1)$ correspond to $(\la_1, \eta_1)$ or $(s, \eta_1)$ and let $\sigma\in\galh_2$.  As before, we have $\eta_1^\sigma=\eta_1$.  Further, by \prettyref{cor:exceptLambda1}, $\Lambda_1$ is $\gal$-equivariant, and certainly $(N_1)_{\la_1}=(N_1)_{\la_1^\sigma}$, hence $\Omega(\chi)^\sigma=\Ind_{(N_1)_{\la_1}}^{N_1}(\Lambda_1(\la_1)^\sigma\eta_1^\sigma)=\Ind_{(N_1)_{\la_1}}^{N_1}(\Lambda_1(\la_1^\sigma)\eta_1)$.  Further, note that $\la_1^\sigma\in\mathcal{E}(T_1, s)^\sigma=\mathcal{E}(T_1, s^\sigma)$ is also completely determined by $s^\sigma$.  Then again here we have $\Omega_1(\chi)^\sigma$ corresponds to the pair $(s^\sigma, \eta_1)$, yielding that $\Omega_1(\chi)^\sigma=\Omega_1(\chi^\sigma)$, as desired.
 \end{proof}
 
  Together, Propositions \ref{prop:d2rational} and \ref{prop:E6} yield the desired result:
 \begin{corollary}
 \prettyref{thm:main}\ref{main2} holds for $G=\type{G}_2(q)$, $\tw{3}\type{D}_4(q)$, $\type{F}_4(q),$ $\type{E}_6^\epsilon(q)_{sc}$, $\type{E}_7(q)_{sc}$, $\type{E}_8(q)$, or
 $\type{B}_n(q)_{sc}$ with $n\geq3$, taking $M=N_1$.
 \end{corollary}

We end by discussing briefly the Sylow $2$-subgroups of the groups under consideration here to arrive at \prettyref{thm:galmck}. By the Corollary of the main theorems of \cite{kondratiev2005}, the Sylow $2$-subgroups of the simple groups listed in \prettyref{thm:galmck}  are self-normalizing.  Further, by \cite{NTreal}, the abelianization $P/[P,P]$ of the Sylow $2$-subgroups in these cases are elementary abelian. 
 
 \begin{proposition}[\cite{NTreal}]\label{prop:NTreal}
 Let $q$ be a power of an odd prime and let $S$ be a simple group $\type{G}_2(q)$, $\tw{3}\type{D}_4(q)$, $\type{F}_4(q),$  $\type{E}_7(q)$, $\type{E}_6^\epsilon(q)$ with $4\mid(q+\epsilon)$, $\type{E}_8(q)$, $\type{C}_n(q)$ with $n\geq 2$,
 $\type{B}_n(q)$ with $n\geq3$, or $\type{D}_n^\epsilon(q)$ with $n\geq 4$.  Let $P$ be a Sylow $2$-subgroup of $S$.  Then $P/[P,P]$ is elementary abelian:
 \end{proposition}
 \begin{proof}
 This is directly from Propositions 3.5, 3.7, and 4.1 of \cite{NTreal}.
 \end{proof}
 
 \begin{proof}[Proof of \prettyref{thm:galmck}]
 By combining \prettyref{prop:NTreal} with \prettyref{prop:d2rational} and the fact that $P=\norm_S(P)$ in the stated cases, we immediately obtain the result, except possibly for $\type{C}_n(q)=\operatorname{PSp}_{2n}(q)$.  However, in the latter case, \cite[Lemma 4.10 and Theorem B]{SFT20}, together with \prettyref{lem:sinvol}, implies that every member of $\irr_{2'}(S)$ is again fixed by $\galh_2$, completing the proof.
 \end{proof}

\section{Acknowledgements}

The author began this work during her stay at the Mathematical
Sciences Research Institute in Berkeley, California as part of the Spring 2018 MSRI semester program
``Group Representation Theory and Applications".  She thanks the MSRI and the organizers of the program for making her stay possible and providing a collaborative and productive work environment.
She would also like to thank 
G. Navarro, B. Sp{\"a}th, and C. Vallejo for helpful discussions and an early preprint of their reduction theorem.  She further thanks C. Vallejo for her comments on an earlier draft of this work.

\bibliographystyle{alpha}
\bibliography{researchreferences}

\end{document}